\documentclass[11pt,reqno]{amsart}
\usepackage{amsmath}
\usepackage{amssymb}
\usepackage{mathrsfs}
\usepackage{microtype}
\usepackage{enumitem}
\usepackage{xcolor}
\usepackage[
  colorlinks=true,
  linkcolor=blue,
  citecolor=blue,
  urlcolor=blue
]{hyperref}

\allowdisplaybreaks
\makeatletter
\renewcommand{\@tocline}[7]{\relax
  \ifnum #1>\c@tocdepth
  \else
    \par\addpenalty\@secpenalty\addvspace{#2}%
    \begingroup\hyphenpenalty\@M
    \@ifempty{#4}{%
      \@tempdima\csname r@tocindent\number#1\endcsname\relax
    }{%
      \@tempdima#4\relax
    }%
    \parindent\z@\leftskip#3\relax\advance\leftskip\@tempdima\relax
    \rightskip\@pnumwidth plus4em\parfillskip-\@pnumwidth
    #5\leavevmode\hskip-\@tempdima #6\nobreak
    \leaders\hbox{$\m@th\mkern4mu\hbox{.}\mkern4mu$}\hfill
    \nobreak\hbox to\@pnumwidth{\@tocpagenum{#7}}\par
    \nobreak
    \endgroup
  \fi}
\renewcommand{\l@section}{\@tocline{1}{5pt}{1pc}{}{}}
\renewcommand{\l@subsection}{\@tocline{2}{0pt}{2.5pc}{5pc}{}}
\makeatother

\newtheorem{theorem}{Theorem}[section]
\newtheorem{proposition}[theorem]{Proposition}
\newtheorem{lemma}[theorem]{Lemma}
\newtheorem{corollary}[theorem]{Corollary}
\newtheorem*{conjecture}{Conjecture}
\theoremstyle{definition}

\newtheorem{example}[theorem]{Example}
\newtheorem{remark}[theorem]{Remark}
\numberwithin{equation}{section}
\newcommand{\HH}{\mathcal H}
\newcommand{\NN}{\mathbb N}
\newcommand{\RR}{\mathbb R}

\newcommand{\PP}{\mathbb P}
\newcommand{\cA}{\mathcal A}
\newcommand{\cS}{\mathcal S}
\newcommand{\cL}{\mathcal L}
\newcommand{\cM}{\mathcal M}
\newcommand{\calQ}{\mathcal Q}
\newcommand{\calD}{\mathcal D}
\newcommand{\dimH}{\dim_{\mathrm H}}
\newcommand{\dimB}{\dim_{\mathrm B}}
\newcommand{\diam}{\operatorname{diam}}
\newcommand{\diag}{\operatorname{diag}}

\begin{document}

\title[Critical Hausdorff Measure for Bedford-McMullen Subshifts]
{Critical Hausdorff Measure for Higher-Dimensional Bedford-McMullen Subshifts}

\author{Hua Qiu}
\address{School of Mathematics, Nanjing University, Nanjing, 210093, P. R. China.}
\thanks{The research of Qiu was supported by the National Natural Science Foundation of China, grants 12471087 and 12531004. The research of Wang was supported by the Nanjing University PhD Student Zhujian Program, grant ZJJH2026B08.}
\email{huaqiu@nju.edu.cn}

\author{Qi Wang}
\address{School of Mathematics, Nanjing University, Nanjing, 210093, P. R. China.}
\email{602023210013@smail.nju.edu.cn}

\subjclass[2020]{Primary 28A80; Secondary 37B10, 37A35, 28A78}
\date{}
\keywords{Bedford--McMullen subshift, weak specification,
critical Hausdorff measure, factor map}

\begin{abstract}
Let $\Lambda=\diag(m_1,\ldots,m_d)$ be a diagonal expanding integer matrix with
$m_1\ge\cdots\ge m_d\ge2$ and at least two distinct diagonal entries, and let
$X\subset\cA^{\NN}$ be a subshift with weak specification over
$\cA=\prod_{r=1}^d\{0,1,\ldots,m_r-1\}$. Let $K=R(X)$ be the corresponding
   self-affine set. Writing $\gamma=\dimH K$, we prove a sharp
critical-measure dichotomy:
\begin{align*}
  \dimH K=\dimB K
  &\quad\Longleftrightarrow\quad
  0<\HH^\gamma(K)<\infty,\\
  \dimH K<\dimB K
  &\quad\Longleftrightarrow\quad
  \HH^\gamma(K)=\infty.
\end{align*}
In the second case, $\HH^\gamma$ is not $\sigma$-finite on $K$.
This resolves Peres' 1994 conjecture for planar systems and extends to higher dimensions.
For sofic subshifts, we prove the existence of the box dimension and show that dimension equality
implies that the critical Hausdorff measure is positive and $\sigma$-finite, although it may be
infinite. We also construct two examples
   with $\dimH K<\dimB K$: in one example the critical Hausdorff measure is positive and
   finite, whereas in the other it is infinite but $\sigma$-finite.
\end{abstract}
\maketitle
\enlargethispage{3pt}

\vspace{-35pt}
\tableofcontents

\section{Introduction}

Fix a diagonal expanding integer matrix
and its digit alphabet
\[
  \Lambda=\diag(m_1,\ldots,m_d),\qquad
  m_1\ge\cdots\ge m_d\ge2,\qquad
  \cA=\prod_{r=1}^d\{0,1,\ldots,m_r-1\}.
\]
Assume that $m_1,\ldots,m_d$ take at least two distinct values. Let
$n_1>\cdots>n_s$ be these distinct values, so that $s\ge2$. Choose
$0=d_0<d_1<\cdots<d_s=d$ so that
\[
  m_r=n_i\qquad d_{i-1}<r\le d_i.
\]
Thus the $i$-th group consists of coordinates $d_{i-1}+1,\ldots,d_i$, and the first $i$
groups together consist of coordinates $1,\ldots,d_i$.
A subshift $X\subset\cA^{\NN}$ is a nonempty closed set satisfying
$\sigma(X)\subseteq X$, where
$\sigma(x_1x_2x_3\cdots)=x_2x_3\cdots$ is the left shift. A \textit{finite word}
of length $N$ over $\cA$ is a finite sequence $I=i_1\cdots i_N\in\cA^N$, and its
\textit{length} is denoted by
$|I|=N$. The empty word has length $0$. If $I=i_1\cdots i_N$ and
$J=j_1\cdots j_M$, their \textit{concatenation} is
\[
  IJ=i_1\cdots i_Nj_1\cdots j_M,
  \qquad |IJ|=N+M.
\]
Write $\cL(X)$ for the set of finite words that occur in sequences of $X$, together with the
empty word. The elements of $\cL(X)$ are called \emph{admissible words}. If
$\cS\subset\cA$ is nonempty and $X=\cS^{\NN}$, then $X$ is the \textit{full shift}
on the digit set $\cS$. The subshift $X$ has
\emph{weak specification} \cite{Feng2011} if
there is an integer $p\ge0$ such that, for every $I,J\in\cL(X)$, there exists an
admissible word $W$ with $|W|\le p$ such that $IWJ\in\cL(X)$. Thus any two admissible words can be joined,
although the joining word may depend on the pair. For
$x=(x_k)_{k\ge1}\in X$, each $x_k\in\cA\subset\mathbb Z^d$ is a $d$-dimensional digit vector.
The associated self-affine set is
\begin{equation}\label{eq:geometric-coding-map}
  K=R(X),\qquad
  R(x)=\sum_{k=1}^{\infty}\Lambda^{-k}x_k.
\end{equation}

We use $\HH^t$, $\dimH$, $\underline{\dim}_{\mathrm B}$, and
$\overline{\dim}_{\mathrm B}$ for Hausdorff measure, Hausdorff dimension, lower box
dimension, and upper box dimension, respectively; see \cite{Falconer2014}.
For a bounded set $E$, when its lower and upper box dimensions agree, their common
value is denoted by $\dimB E$.
For a nonempty set $E$, we call $\HH^{\dimH E}(E)$ its \emph{critical Hausdorff
measure}.
\smallskip

The Hausdorff dimension gives the threshold at which Hausdorff measure changes from infinity
to zero, but it does not determine the measure at the threshold itself. More precisely, if
$\gamma=\dimH E$, then $\HH^t(E)=\infty$ for $t<\gamma$ and $\HH^t(E)=0$ for
$t>\gamma$, whereas $\HH^\gamma(E)$ may be zero, positive and finite, or infinite. For a
self-similar set satisfying the open set condition, Hutchinson's theorem gives
$0<\HH^\gamma(E)<\infty$ \cite{Hutchinson1981}. The situation is more delicate for the
diagonal self-affine sets considered here: symbolic cylinders have different side lengths in
different coordinate directions, so their distribution among the coordinate fibres contains
geometric information that is not recorded by the Hausdorff dimension alone. Thus the
critical Hausdorff measure is a genuinely finer quantity than the dimension.

The Bedford--McMullen family is the basic non-conformal setting in which this distinction
can be seen explicitly. When $d=2$ and $X=\cS^{\NN}$ is a full shift, the set $K$ is a
\emph{Bedford--McMullen carpet}.
In 1984, Bedford and McMullen independently computed its Hausdorff and box dimensions
\cite{Bedford1984,McMullen1984}. If all nonempty fibres of the projection of $\cS$ onto
the second coordinate have the same cardinality, then
$\dimH K=\dimB K$ and the critical Hausdorff measure is positive and finite.
McMullen asked about the critical measure when these dimensions are different. In 1994,
Peres answered this question by proving that, in that case, the
critical Hausdorff measure is infinite and not $\sigma$-finite \cite{Peres1994}. Thus, for a planar
Bedford-McMullen full shift and $\gamma=\dimH K$,
\begin{align*}
  \dimH K=\dimB K
  &\quad\Longleftrightarrow\quad
  0<\HH^\gamma(K)<\infty,\\
  \dimH K<\dimB K
  &\quad\Longleftrightarrow\quad
  \HH^\gamma(K)=\infty.
\end{align*}

At the end of the same paper, Peres considered sets coded by a primitive transition
matrix and formulated the following conjecture. Let $\cS\subset\cA$ be a digit set and let
$A=(A_{uv})_{u,v\in\cS}$ be a primitive zero-one matrix; that is, $A^q$ is a positive
matrix for some integer $q\ge1$. The corresponding subshift is
\[
  X_A=\{x=(x_k)_{k\ge1}\in\cS^{\NN}:A_{x_kx_{k+1}}=1\text{ for every }k\ge1\},
  \qquad K_A=R(X_A).
\]
The subshift $X_A$ is called a \emph{primitive subshift of finite type}.

\begin{conjecture}[\cite{Peres1994}]
If $\gamma_A=\dimH K_A<\dimB K_A$, then
$\HH^{\gamma_A}(K_A)=\infty$.
\end{conjecture}

The motivation for this paper is to prove this two-dimensional conjecture and to extend it
to higher-dimensional diagonal affine-invariant sets coded by subshifts satisfying weak
specification.
The main result is the following.

\begin{theorem}\label{thm:main}
Let $X\subset\cA^{\NN}$ be a subshift satisfying weak specification and set $K=R(X)$. Let
$\gamma=\dimH K$. Then
\begin{align}
  \dimH K=\dimB K
  &\quad\Longleftrightarrow\quad
  0<\HH^\gamma(K)<\infty,
  \label{eq:finite-measure-characterization}\\
  \dimH K<\dimB K
  &\quad\Longleftrightarrow\quad
  \HH^\gamma(K)=\infty.
  \label{eq:infinite-measure-characterization}
\end{align}
In the second case, $\HH^\gamma$ is not $\sigma$-finite on $K$.
\end{theorem}

Every primitive subshift of finite type satisfies weak specification, so
Theorem~\ref{thm:main} resolves Peres' conjecture. Under weak specification, Z. Feng proved
the existence of the box dimension and that dimension equality implies positive and finite
critical Hausdorff measure \cite{Feng2026}. Our contribution is the converse measure statement: a strict dimension gap
implies that the critical Hausdorff measure is infinite and, in fact, not $\sigma$-finite.
In particular, the critical Hausdorff measure can never vanish under weak specification.
\smallskip

The dimension theory surrounding this problem developed in several directions. In the
plane, the constructions of Lalley--Gatzouras and Bara\'nski allow more general rectangular
patterns and contraction ratios than the Bedford--McMullen construction
\cite{GatzourasLalley1992,Baranski2007}. For higher-dimensional Bedford--McMullen sets,
Kenyon and Peres established the Hausdorff-dimension variational principle and studied
measures of full dimension \cite{KenyonPeres1996,KenyonPeres1996a}. In the present setting,
approximate cubes encode different coordinate groups over different proportions of symbolic
time, so the Hausdorff-dimension variational principle involves a weighted sum of factor
entropies. D.-J. Feng's weighted thermodynamic formalism provides the corresponding
equilibrium states and entropy estimates \cite{Feng2011}.

Several finer aspects of diagonal self-affine geometry have also been studied, including
dimension gaps for sponges, Assouad-type dimensions, $L^q$ spectra, random sponges,
and dimension formulas for diagonal self-affine sets and measures
\cite{DasSimmons2017,FraserHowroyd2017,Kolossvary2023,BarralBrunet2025,Feng2025,
Rapaport2026}. The critical Hausdorff measure is a different refinement: it asks for the
size of a set at its already known Hausdorff dimension. For Bara\'nski carpets, the
authors obtained a dichotomy in terms of uniform fibre conditions \cite{QiuWang2024}.
Related recent work on Hausdorff measure and finer geometry for planar self-affine sets
includes \cite{Barany2026,BaranyKaenmakiYu2026}.
\medskip

To understand how the conclusions in Theorem~\ref{thm:main} can change when weak specification fails, we next
consider sofic subshifts, which contain all subshifts of finite type. A sofic subshift
$X\subset\cA^{\NN}$ can be represented by a finite
vertex-labelled directed graph and is therefore closely connected with graph-directed
constructions; compare \cite{MauldinWilliams1988}. It is \emph{irreducible} if, for every
$I,J\in\cL(X)$, there is a word $W\in\cL(X)$ such that $IWJ\in\cL(X)$. An irreducible
sofic subshift has a finite strongly connected graph presentation. For each ordered
pair $(u,v)$ of vertices, choose a directed path of positive length from $u$ to $v$. Since
there are only finitely many such pairs, the lengths of these paths
have a finite maximum; denote it by $p$. Given
$I,J\in\cL(X)$, choose paths whose labels are $I$ and $J$. If $u$ is the terminal vertex of
the first path and $v$ is the initial vertex of the second, let $W$ be the word formed by the
labels of the vertices between $u$ and $v$ on the chosen path. Then $|W|\le p$ and
$IWJ\in\cL(X)$. Thus every irreducible sofic subshift satisfies weak specification and is covered by
Theorem~\ref{thm:main}. We therefore
focus on reducible sofic subshifts, which need not satisfy weak specification.

\begin{theorem}\label{thm:reducible}
For every sofic subshift $X$, the box dimension of $K=R(X)$ exists. Write
$\gamma=\dimH K$. Moreover:
\begin{enumerate}[label=\textup{(\roman*)}]
  \item If $\dimH K=\dimB K$, then $\HH^\gamma(K)>0$ and $\HH^\gamma$ is
  $\sigma$-finite on $K$, but $\HH^\gamma(K)$ need not be finite.
  \item If $\dimH K<\dimB K$, then $\HH^\gamma(K)$ may be positive and finite, or it may
  be infinite and $\sigma$-finite.
\end{enumerate}
\end{theorem}
\begin{remark}
For a sofic subshift, Proposition~\ref{prop:reducible-box} not only proves
that the box dimension exists, but also gives its formula in terms of the topological
entropies of the strongly connected components and their factor subshifts.
\end{remark}

Together, these results compare the weak-specification setting with the reducible
sofic setting. Under weak specification, Theorem~\ref{thm:main} gives a complete
dichotomy for the critical Hausdorff measure according to whether the dimensions are equal. For
sofic subshifts, dimension equality implies positivity and
$\sigma$-finiteness but not finiteness. When the dimensions differ, both finite and infinite
critical Hausdorff measure can occur.

\smallskip
\noindent\textbf{Structure of the manuscript.}
Section~\ref{sec:preliminaries} introduces the notation and reduces the second implication in
Theorem~\ref{thm:main} to
Propositions~\ref{thm:pressure-to-measure} and~\ref{thm:pressure-positive}.
Section~\ref{sec:cubes} defines approximate cubes and proves the required density estimates.
Section~\ref{sec:word-partitions} proves the partition lemma for finite admissible word sets.
Section~\ref{sec:measure-construction} uses this partition to construct a probability measure
and prove Proposition~\ref{thm:pressure-to-measure}. Section~\ref{sec:pressure-inequality} proves
Proposition~\ref{thm:pressure-positive}. Finally, Section~\ref{sec:reducible} proves
Theorem~\ref{thm:reducible} and examines the dimensions and critical Hausdorff measure of
reducible sofic subshifts.

\section{Preliminaries and reduction of the main theorem}\label{sec:preliminaries}

Throughout, $\log$ denotes the natural logarithm. For $x\in\RR$,
$\lfloor x\rfloor$ is the greatest integer less than or equal to $x$, and $\lceil x\rceil$
is the least integer greater than or equal to $x$.
For positive quantities $A$ and $B$, the notation $A\lesssim B$ means that
$A\le CB$ for a constant $C>0$, $A\gtrsim B$ means that $B\lesssim A$, and
$A\asymp B$ means that both $A\lesssim B$ and $B\lesssim A$.
For a real quantity $F$ and a positive quantity $B$, the notation $F=O(B)$ means
$|F|\lesssim B$, while $F=o(B)$ means that $F/B\to0$ in the stated limit. A subscript on
$\lesssim$, $\gtrsim$, $\asymp$, or $O$ lists the parameters on which the corresponding constant may
depend.
For a finite set $E$, the symbol $\#E$ denotes its cardinality.

\subsection{Factor maps and weighted pressure}

Throughout Sections~\ref{sec:preliminaries}--\ref{sec:pressure-inequality},
$X\subset\cA^{\NN}$ denotes a subshift satisfying weak specification, $K=R(X)$, and
$\cS=\{x_1:x\in X\}$ is the set of digits that occur in $X$. For
$x=(x_j)_{j\ge1}$, write $x|N=x_1\cdots x_N$. If $I=i_1\cdots i_N$ and
$0\le a\le b\le N$, write $I|_a^b=i_{a+1}\cdots i_b$, abbreviate
$I|b=I|_0^b$, and write $|I|=N$. We use the convention $I|_a^a=\varnothing$.
If $Y$ is a subshift over a finite alphabet,
let $\cL_0(Y)=\{\varnothing\}$ and
\[
  \cL_N(Y)=\{y_1\cdots y_N:y=(y_j)_{j\ge1}\in Y\}\quad(N\ge1),
  \qquad
  \cL(Y)=\bigcup_{N\ge0}\cL_N(Y).
\]
For $N\ge1$ and $I\in\cL_N(Y)$, the cylinder generated by $I$ is
$[I]=\{y\in Y:y|N=I\}$.

We write $\cM_\sigma(Y)$ for the set of invariant Borel probability measures on $Y$. For a
Borel probability measure $\nu$ on $Y$, set
\begin{equation}\label{eq:block-entropy-definition}
  H(\nu|\cL_N(Y))=-\sum_{I\in\cL_N(Y)}\nu[I]\log\nu[I],
\end{equation}
where $0\log0=0$. For $\mu\in\cM_\sigma(Y)$, its measure-theoretic entropy is
\[
  h(\mu)=\lim_{N\to\infty}\frac1N H(\mu|\cL_N(Y)),
\]
whereas the topological entropy of the whole subshift is
\[
  h_{\mathrm{top}}(Y)=\lim_{N\to\infty}\frac1N\log\#\cL_N(Y).
\]
Both limits exist by subadditivity and Fekete's lemma, since
$H(\mu|\cL_{N+M}(Y))\le H(\mu|\cL_N(Y))+H(\mu|\cL_M(Y))$ and
$\#\cL_{N+M}(Y)\le\#\cL_N(Y)\,\#\cL_M(Y)$.
In
particular, $h(\mu)\le h_{\mathrm{top}}(Y)$.
\smallskip

Recall that there are integers $n_1>\cdots>n_s\ge2$, where $s\ge2$, and
$0=d_0<d_1<\cdots<d_s=d$ such that
\[
  m_r=n_i\qquad d_{i-1}<r\le d_i.
\]
Set $n_0=\infty$, $\cA_1=\cA$, and
\[
  \cA_i=\prod_{r>d_{i-1}}\{0,\ldots,m_r-1\}
  \qquad 2\le i\le s.
\]
Let $\pi_1:\cA\to\cA_1$ be the identity.  For $2\le i\le s$, let
$\pi_i:\cA_{i-1}\to\cA_i$ delete the coordinates numbered at most $d_{i-1}$.
The maps are applied coordinatewise to sequences and words.  Write
\begin{equation}\label{eq:factor-chain}
  \tau_i=\pi_i\circ\cdots\circ\pi_1,
  \qquad X_i=\tau_i(X),
  \qquad
  X=X_1\xrightarrow{\pi_2}X_2\xrightarrow{\pi_3}\cdots
  \xrightarrow{\pi_s}X_s.
\end{equation}
We call the sequence of factor maps in \eqref{eq:factor-chain} the \emph{factor chain}.

If $X$ satisfies weak specification with bound $p$, then every factor $X_i$ satisfies
weak specification with the same bound. To see this, let
$U,V\in\cL(X_i)$, choose $I,J\in\cL(X)$ such that $\tau_i(I)=U$ and $\tau_i(J)=V$.
There is $W\in\cL(X)$ such that $|W|\le p$ and $IWJ\in\cL(X)$. Hence
\[
  U\,\tau_i(W)\,V=\tau_i(IWJ)\in\cL(X_i),
  \qquad |\tau_i(W)|=|W|\le p.
\]

If $\mu$ is a Borel probability measure on $X$, its
push-forward under $\tau_i$ is the probability measure $\tau_i\mu=(\tau_i)_*\mu$ on $X_i$
given by
\[
  (\tau_i\mu)(B)=\mu(\tau_i^{-1}(B))
  \qquad\text{for every Borel set }B\subset X_i.
\]
If $\mu$ is invariant under $\sigma$, then $\tau_i\mu$ is also invariant, as
$\tau_i\circ\sigma=\sigma\circ\tau_i$.

Set
\begin{equation}\label{eq:theta-a}
  \theta_0=0,
  \qquad \theta_i=\frac{\log n_s}{\log n_i},
  \qquad a_i=\theta_i-\theta_{i-1},
  \qquad \mathbf a=(a_1,\ldots,a_s).
\end{equation}
Then $0<\theta_1<\cdots<\theta_s=1$.

If
$\mathbf b=(b_1,\ldots,b_s)$, where $b_i>0$ and $\sum_i b_i=1$, write
\begin{equation}\label{eq:weighted-pressure}
  P^{\mathbf b}(\sigma)
  =\sup_{\mu\in\cM_\sigma(X)}
  \sum_{i=1}^s b_i h(\tau_i\mu).
\end{equation}
$P^{\mathbf b}(\sigma)$ is called the
\textit{$\mathbf b$-weighted pressure} of $\sigma$. In the notation of D.-J. Feng \cite{Feng2011},
this is $P^{\mathbf b}(\sigma,0)$; we omit the zero potential from the notation. An
invariant measure attaining the supremum in \eqref{eq:weighted-pressure} is called a
\textit{$\mathbf b$-weighted equilibrium state} \cite{Feng2011}.

We shall use the following results from Kenyon and Peres \cite{KenyonPeres1996a},
D.-J. Feng \cite{Feng2011}, and Z. Feng \cite{Feng2026}.

\begin{proposition}\label{prop:known-results}
Let $X\subset\cA^{\NN}$ be a subshift satisfying weak specification and set $K=R(X)$.
Then:
\begin{enumerate}[label=\textup{(\roman*)}]
  \item
  \begin{equation}\label{eq:dimension-formulas}
    \dimH K=\frac{P^{\mathbf a}(\sigma)}{\log n_s},
    \qquad
    \dimB K=\frac{1}{\log n_s}
      \sum_{i=1}^s a_i h_{\mathrm{top}}(X_i).
  \end{equation}
  \item Every positive weight vector $\mathbf b$ has a unique
  $\mathbf b$-weighted equilibrium state $\eta_{\mathbf b}$.
  \item For each fixed $\mathbf b$ and $1\le i\le s$,
  \begin{equation}\label{eq:block-entropy}
    H\bigl((\tau_i\eta_{\mathbf b})|\cL_N(X_i)\bigr)
     =Nh(\tau_i\eta_{\mathbf b})+O_{\mathbf b}(1).
  \end{equation}
  \item If $\gamma=\dimH K$ and $\dimH K=\dimB K$, then
  \begin{equation}\label{eq:feng-finite-critical-measure}
    0<\HH^\gamma(K)<\infty.
  \end{equation}
\end{enumerate}
\end{proposition}

\begin{proof}
The Ledrappier--Young formula and the full-dimension variational principle in
\cite[Theorem~1.1 and Lemma~4.3]{KenyonPeres1996a} give
\[
  \dimH K
  =\sup_{\mu\in\cM_\sigma(X)}
    \sum_{i=1}^s
    \left(\frac{1}{\log n_i}-\frac{1}{\log n_{i-1}}\right)
    h(\tau_i\mu)
  =\frac{1}{\log n_s}
    \sup_{\mu\in\cM_\sigma(X)}\sum_{i=1}^s a_i h(\tau_i\mu)
  =\frac{P^{\mathbf a}(\sigma)}{\log n_s}.
\]
Since $X$ satisfies weak specification, \cite[Proposition~3.3]{Feng2026} shows that the box
dimension of $K$ exists and gives
\[
  \dimB K
  =\frac{1}{\log n_s}\sum_{i=1}^s a_i h_{\mathrm{top}}(X_i).
\]
This proves \textup{(i)}.

Conclusion \textup{(ii)} follows directly from \cite[Theorem~1.2]{Feng2011}, applied to the
zero potential.

We next prove \textup{(iii)}. For $i=1$, equation \eqref{eq:block-entropy} is precisely the
entropy estimate in \cite[Theorem~7.3(iii)]{Feng2011}. Let $2\le i\le s$. The factor $X_i$
satisfies weak specification.
Applying \cite[Lemma~7.2(ii)]{Feng2011} successively to $\pi_2,\ldots,\pi_i$ shows that
$\tau_i\eta_{\mathbf b}$ is the weighted equilibrium state of the induced potential on
the factor chain $X_i\to\cdots\to X_s$.
By \cite[Lemma~5.7]{Feng2011}, each recursively induced word function remains in
the class required by \cite[Theorems~5.5 and 7.3]{Feng2011}.
This induced potential need not be zero. If $i<s$, its weight vector is
$\left(\sum_{\ell=1}^i b_\ell,b_{i+1},\ldots,b_s\right)$, so
\cite[Theorem~7.3(iii)]{Feng2011} applies. For $i=s$, \cite[Theorem~5.5]{Feng2011} applies. In either case,
\[
  \sum_{J\in\cL_N(X_i)}(\tau_i\eta_{\mathbf b})[J]
       \log(\tau_i\eta_{\mathbf b})[J]
   =-Nh(\tau_i\eta_{\mathbf b})+O_{\mathbf b}(1).
\]
Multiplying by $-1$ and using \eqref{eq:block-entropy-definition} gives
\eqref{eq:block-entropy}.

Finally, under the assumption $\dimH K=\dimB K$, conclusion \textup{(i)} of
\cite[Theorem~1.1]{Feng2026} gives
$
  0<\HH^{\dimH K}(K)<\infty.
$ Taking $\gamma=\dimH K$ proves \textup{(iv)}.
\end{proof}

\subsection{A reduction of Theorem~\ref{thm:main}}

Sections~\ref{sec:cubes}--\ref{sec:pressure-inequality} prove the second implication in
Theorem~\ref{thm:main}. For $0<\beta\le1$, write
\begin{equation}\label{eq:power-weight}
  a_i(\beta)=\theta_i^\beta-\theta_{i-1}^\beta,
  \qquad
  \mathbf a(\beta)=(a_1(\beta),\ldots,a_s(\beta)).
\end{equation}
Note that $\mathbf a(1)=\mathbf a$, where $\mathbf a$ is defined in
\eqref{eq:theta-a}.
Since $0=\theta_0<\theta_1<\cdots<\theta_s=1$,
\[
  a_i(\beta)>0,
  \qquad
  \sum_{i=1}^s a_i(\beta)=1.
\]
Thus $\mathbf a(\beta)$ is a positive weight vector.
Let $\eta=\eta_{\mathbf a}$, and write
\begin{equation}\label{eq:pressure-gap}
  c(\beta)=P^{\mathbf a(\beta)}(\sigma)
  -\sum_{i=1}^s a_i(\beta)h(\tau_i\eta)\ge0.
\end{equation}
The last inequality follows directly from the variational definition of
$P^{\mathbf a(\beta)}(\sigma)$.

Given a continuous increasing function $\phi:[0,\infty)\to[0,\infty)$ with $\phi(0)=0$,
called a \emph{gauge function}, the associated Hausdorff measure is
\[
  \HH^\phi(E)=\lim_{\delta\downarrow0}
  \inf\left\{\sum_j \phi(\diam U_j):
  E\subset\bigcup_jU_j,\ \diam U_j\le\delta\right\}.
\]
For $t>0$, taking $\phi(r)=r^t$ gives the usual Hausdorff measure $\HH^t$.

We now reduce the second implication in Theorem~\ref{thm:main} to two propositions. The first
shows that $c(\beta)>0$ gives the required Hausdorff-measure conclusion, and the second shows
that the dimension gap $\dimH K<\dimB K$ gives $c(\beta)>0$. Their proofs are given in
Sections~\ref{sec:measure-construction} and~\ref{sec:pressure-inequality}, respectively.
Together with Proposition~\ref{prop:known-results}\textup{(iv)}, these two propositions prove
Theorem~\ref{thm:main}.

\begin{proposition}\label{thm:pressure-to-measure}
Let $X\subset\cA^{\NN}$ be a subshift satisfying weak specification, set $K=R(X)$, and let
$\gamma=\dimH K$.
If $c(\beta)>0$ for some $\beta\in(1/2,1)$, then there is $\varepsilon>0$ such that the
gauge function $g_\beta$, satisfying
$g_\beta(r)=r^\gamma\exp\{-\varepsilon(\log(1/r))^\beta\}$ for $0<r\le1$, has
\begin{equation*}
  \HH^{g_\beta}(K)=\infty,
  \qquad
  \HH^\gamma(K)=\infty.
\end{equation*}
Neither $\HH^{g_\beta}$ nor $\HH^\gamma$ is $\sigma$-finite on $K$.
\end{proposition}

\begin{proposition}\label{thm:pressure-positive}
Let $X\subset\cA^{\NN}$ be a subshift satisfying weak specification and set $K=R(X)$. If
$\dimH K<\dimB K$, then
$c(\beta)>0$ for some $\beta\in(1/2,1)$.
\end{proposition}

\begin{proof}[Proof of Theorem~\ref{thm:main} based on Propositions~\ref{thm:pressure-to-measure} and~\ref{thm:pressure-positive}]
Suppose first that $\dimH K=\dimB K$. Then
\eqref{eq:feng-finite-critical-measure} gives
\[
  0<\HH^\gamma(K)<\infty.
\]
Suppose next that $\dimH K<\dimB K$. Proposition~\ref{thm:pressure-positive} gives
$\beta\in(1/2,1)$ such that $c(\beta)>0$. Proposition~\ref{thm:pressure-to-measure} then gives
\[
  \HH^\gamma(K)=\infty,
\]
and shows that $\HH^\gamma$ is not $\sigma$-finite on $K$.

The general inequality $\dimH K\le\dimB K$ shows that these are the only two possible cases.
The finite-measure conclusion in the first case and the infinite-measure conclusion in the
second case prove both converse implications in
\eqref{eq:finite-measure-characterization} and
\eqref{eq:infinite-measure-characterization}.
\end{proof}

\begin{remark}
	The proof of Proposition~\ref{thm:pressure-positive} gives a stronger conclusion. If
	$\dimH K<\dimB K$, then every nonempty open interval $I\subset(0,1)$ contains some
	$\beta$ for which $c(\beta)>0$. Otherwise, the function $c$ would vanish throughout $I$, and the argument
	in Section~\ref{sec:pressure-inequality} would imply $\dimH K=\dimB K$. Consequently, for
	every $\beta_0\in(1/2,1)$, there is some $\beta\in(\beta_0,1)$ with $c(\beta)>0$.
	Proposition~\ref{thm:pressure-to-measure} then shows that the exponent $\beta$ in the
	gauge function $g_\beta$ may be chosen arbitrarily close to $1$.
\end{remark}

\section{Approximate cubes and the density estimates}\label{sec:cubes}

This section introduces the approximate cubes used throughout the proof. We describe both
their geometric form and their symbolic representation. We then use a density estimate to
deduce infinite and non-$\sigma$-finite Hausdorff measure from sufficiently small
approximate-cube masses.

For a word $I=i_1\cdots i_N$ and $0\le k\le N$, write
\[
  R_k(I)=\sum_{\ell=1}^k\Lambda^{-\ell}i_\ell,
  \qquad R_0(I)=0.
\]
Thus
\[
  R(x)=R_k(x|k)+\Lambda^{-k}R(\sigma^kx).
\]
For $I\in\cA^k$, the \emph{level-$k$ approximate cube} is
\begin{equation*}
  Q_k(I)=\prod_{i=1}^s\prod_{r=d_{i-1}+1}^{d_i}
  \left[
    R_{\lfloor\theta_i k\rfloor}(I)_r,
    R_{\lfloor\theta_i k\rfloor}(I)_r+n_i^{-\lfloor\theta_i k\rfloor}
  \right),
\end{equation*}
with the adjustment that a factor is closed on the right whenever its right endpoint is $1$.
Here $R_{\lfloor\theta_i k\rfloor}(I)_r$ denotes the $r$-th coordinate of
$R_{\lfloor\theta_i k\rfloor}(I)$. Note that the level-$k$ approximate cubes form a
partition of $[0,1]^d$. These partitions are nested:
every level-$(k+1)$ approximate cube is contained in a unique level-$k$ approximate cube.
Every side length of $Q_k(I)$ is comparable to $n_s^{-k}$; hence there are constants
$c_0,C_0>0$, independent of $k$ and $I$, such that
\begin{equation}\label{eq:cube-size}
  c_0n_s^{-k}\le\diam Q_k(I)\le C_0n_s^{-k}.
\end{equation}
For $x\in\cA^{\NN}$, write $Q_k(x)=Q_k(x|k)$.

For $x\in X$, let the corresponding word set be
\begin{equation}\label{eq:symbolic-cube}
  \Gamma_k(x)=\left\{I\in\cL_k(X):
  \tau_i(I|\lfloor\theta_i k\rfloor)
  =\tau_i(x|\lfloor\theta_i k\rfloor)
  \text{ for }1\le i\le s\right\}.
\end{equation}
The corresponding \textit{cylinder set} is
\[
  [\Gamma_k(x)]=\bigcup_{I\in\Gamma_k(x)}[I].
\]
The following elementary observation from Euclidean geometry will be used: for every
level-$k$ approximate cube $Q$ with $R^{-1}(Q)\cap X\ne\varnothing$, there are sequences
$x^{(1)},\ldots,x^{(L)}\in X$ such that
\begin{equation}\label{eq:encoding-interface}
  R^{-1}(Q)\cap X
  \subseteq\bigcup_{\ell=1}^{L}[\Gamma_k(x^{(\ell)})],
  \qquad L\le2^d.
\end{equation}
Set
$
  \calQ_k(X)=\{Q_k(x):x\in X\}.
$
For every $x\in X$, the point $R(x)$ belongs to the closure of $Q_k(x)$. Consequently,
the closures of the approximate cubes in $\calQ_k(X)$ cover $K$.
For $z\in[0,1]^d$, let $Q_k(z)$ be the unique level-$k$ approximate cube containing $z$.
If $z=R(x)$ is not on a grid boundary, then $Q_k(z)=Q_k(x)$.
Let
\begin{equation}\label{eq:word-vectors}
\begin{aligned}
  \calD_k(X)=\biggl\{(J_1,\ldots,J_s)\in
  \prod_{i=1}^s\cL_{\lfloor\theta_i k\rfloor-
  \lfloor\theta_{i-1}k\rfloor}(X_i):{}&
  \text{there is }I\in\cL_k(X)\text{ such that}\\
  &\tau_i(I|_{\lfloor\theta_{i-1}k\rfloor}^{\lfloor\theta_i k\rfloor})=J_i
  \text{ for every }i\biggr\}.
\end{aligned}
\end{equation}
By \cite[Lemma~2.7]{Feng2026},
\begin{equation}\label{eq:number-approximate-cubes}
  \#\calQ_k(X)=\#\calD_k(X).
\end{equation}

\begin{lemma}\label{lem:density}
Let $\lambda$ be a Borel probability measure supported on $K$, and let $g$ be a gauge function.
If
\[
  \limsup_{k\to\infty}
  \frac{\lambda(Q_k(z))}{g(n_s^{-k})}=0
  \quad\text{for $\lambda$-almost every }z,
\]
then $\HH^g(K)=\infty$. Moreover, $\HH^g$ is not $\sigma$-finite on $K$.
\end{lemma}

\begin{proof}
For every $C>0$, the assumed upper density is at most $C$ almost everywhere. Hence
\cite[Lemma~2.9(i)]{Feng2026} gives $\HH^g(K)\gtrsim C^{-1}$, where the implicit constant
is independent of $C$. Letting $C\downarrow0$ proves $\HH^g(K)=\infty$.
Let $E\subset K$ be Borel with $\lambda(E)>0$ and define
$\lambda_E(A)=\lambda(A\cap E)/\lambda(E)$. For $\lambda_E$-almost every $z$,
\[
  \frac{\lambda_E(Q_k(z))}{g(n_s^{-k})}
  \le\frac1{\lambda(E)}\frac{\lambda(Q_k(z))}{g(n_s^{-k})}
  \longrightarrow0.
\]
Applying the same argument to $\lambda_E$ gives $\HH^g(E)=\infty$.
Thus every Borel set of finite $\HH^g$ measure has zero $\lambda$-mass. By Borel regularity
of Hausdorff measure, every $\HH^g$-measurable set of finite measure is contained in a Borel
set of the same finite measure. Therefore no countable family of finite-$\HH^g$-measure
sets can cover $K$, and $\HH^g$ is not $\sigma$-finite on $K$.
\end{proof}

\section{Partition of admissible words}\label{sec:word-partitions}

This section proves the finite-word decomposition used in the proof of
Proposition~\ref{thm:pressure-to-measure}. For each $N$, we partition $\cL_N(X)$ into
$O(N^{s-1})$ sets having uniform fibres at every factor level.
\smallskip

The idea comes from a standard approximation method for planar self-affine carpets. For a
Bedford--McMullen carpet, one can pass to a sufficiently high iterate and
select words whose digit frequencies are close to those of a measure of full Hausdorff
dimension. The selected words define a Bedford--McMullen subcarpet with uniform fibres and
Hausdorff dimension arbitrarily close to that of the original carpet. This approximation from
within was used in \cite[Section~4]{FergusonJordanShmerkin2010} and
\cite[Section~6]{FraserShmerkin2016}. Here a single subsystem is not enough. Instead, for
each fixed word length, we divide all admissible words into finitely many sets, each having
uniform fibres at every factor level.
\smallskip

Fix $N\in\NN$ and write $\mathcal W_i=\cL_N(X_i)$. When no confusion can arise, we use the
same symbols for the restrictions
\[
  \pi_{i+1}:\mathcal W_i\to\mathcal W_{i+1},
  \qquad
  \tau_i:\mathcal W_1\to\mathcal W_i.
\]
For $C\subset\mathcal W_1$ and $J\in\tau_i(C)$, write
\[
  m_i(C,J)=\#\{I\in C:\tau_i(I)=J\}.
\]
We consider sets $C\subset\mathcal W_1$ for which, for every
$1\le i\le s$, the number $m_i(C,J)$ is independent of
$J\in\tau_i(C)$. In other words, all fibres of the map
$\tau_i:C\to\tau_i(C)$ have the same cardinality. We say that such a set $C$ has
\emph{uniform fibres at every level}. Denote the common fibre cardinality by $M_i(C)$, and write
$B_i(C)=\#\tau_i(C)$. Then
\begin{equation}\label{eq:partition-cardinality}
	\#C=B_i(C)M_i(C)\qquad 1\le i\le s.
\end{equation}
Recall that $\cS=\{x_1:x\in X\}$, and write $\cS_i=\tau_i(\cS)$. A partition of
$\mathcal W_1$ is a family of pairwise disjoint subsets whose union is $\mathcal W_1$.

\begin{lemma}\label{lem:word-partition}
	The set $\mathcal W_1$ has a partition
	$\{C_t\}_{t=1}^{q_N}$ such that every $C_t$ has uniform fibres at every level. Moreover,
	\begin{equation}\label{eq:number-partition-parts}
		q_N\le
		\prod_{i=1}^{s-1}
		\left(1+N\left\lceil
          \frac{\log\#\cS_i}{\log2}\right\rceil\right)
		=O(N^{s-1}).
	\end{equation}
\end{lemma}

\begin{proof}
We construct the partition recursively along the factor chain. The construction will use
sets of the form
\[
  G_i(r_1,\ldots,r_{i-1};J_i)\subset\mathcal W_1,
  \qquad J_i\in\mathcal W_i,\quad 2\le i\le s.
\]
For each $i$, we construct the nonempty sets of this form so that they satisfy the following
three properties.
\begin{enumerate}[label=\textup{(P\arabic*)}]
  \item They are pairwise disjoint and their union is $\mathcal W_1$.
  \item Every word in $G_i(r_1,\ldots,r_{i-1};J_i)$ projects to $J_i$ under $\tau_i$.
  \item For $1\le \ell\le i$ and
  $L\in\tau_\ell(G_i(r_1,\ldots,r_{i-1};J_i))$,
  \begin{equation}\label{eq:inductive-equal-fibres}
    \#\{I\in G_i(r_1,\ldots,r_{i-1};J_i):\tau_\ell(I)=L\}
    =2^{r_1+\cdots+r_{\ell-1}}.
  \end{equation}
  For $\ell=1$, the exponent is the empty sum $0$, so the right-hand side is $1$.
\end{enumerate}
In particular, taking $\ell=i$ in \eqref{eq:inductive-equal-fibres} and using
\textup{(P2)} gives
\begin{equation}\label{eq:inductive-group-size}
  \#G_i(r_1,\ldots,r_{i-1};J_i)
  =2^{r_1+\cdots+r_{i-1}}.
\end{equation}

We first construct the level-$2$ sets. For $J_2\in\mathcal W_2$, let
\[
  q_1(J_2)=\#\{I\in\mathcal W_1:\tau_2(I)=J_2\}\geq 1.
\]
Write its binary expansion as
\begin{equation*}
  q_1(J_2)=\sum_{r_1\ge0}\varepsilon_{r_1}(J_2)2^{r_1},
  \qquad \varepsilon_{r_1}(J_2)\in\{0,1\}.
\end{equation*}
Using this binary expansion, split the fibre $\tau_2^{-1}(J_2)$ into sets
\[
  G_2(r_1;J_2),
  \qquad \#G_2(r_1;J_2)=2^{r_1},
\]
for each $r_1$ with $\varepsilon_{r_1}(J_2)=1$. The cardinalities of these sets add up to
$q_1(J_2)$. As $J_2$ varies, these sets
partition $\mathcal W_1$. Every word $I\in G_2(r_1;J_2)$ satisfies $\tau_2(I)=J_2$, and
$\tau_1$ is the identity.
Consequently, properties \textup{(P1)}--\textup{(P3)} hold for $i=2$.
\smallskip

Suppose now that the level-$i$ sets have been constructed for some $2\le i<s$. Fix
$(r_1,\ldots,r_{i-1})$ and $J_{i+1}\in\mathcal W_{i+1}$, and consider the finite family
\begin{equation}\label{eq:family-to-be-grouped}
  \mathscr F_i(r_1,\ldots,r_{i-1};J_{i+1})
  =\left\{
    G_i(r_1,\ldots,r_{i-1};J_i):
    \begin{array}{l}
      J_i\in\mathcal W_i,\ \pi_{i+1}(J_i)=J_{i+1},\\
      G_i(r_1,\ldots,r_{i-1};J_i)\ne\varnothing
    \end{array}
  \right\}.
\end{equation}
Let $q_i(r_1,\ldots,r_{i-1};J_{i+1})$ be the number of sets in this family. If this
number is zero, there is nothing to do. Otherwise, write its binary expansion as
\begin{equation*}
  q_i(r_1,\ldots,r_{i-1};J_{i+1})
  =\sum_{r_i\ge0}\varepsilon_{r_i}2^{r_i},
  \qquad \varepsilon_{r_i}\in\{0,1\}.
\end{equation*}
Partition the family in \eqref{eq:family-to-be-grouped} into disjoint subfamilies
$\mathscr F_{i,r_i}$ with $\varepsilon_{r_i}=1$, so that
$\#\mathscr F_{i,r_i}=2^{r_i}$. Let
\begin{equation}\label{eq:next-level-group}
  G_{i+1}(r_1,\ldots,r_i;J_{i+1})
  =\bigcup_{G\in\mathscr F_{i,r_i}}G.
\end{equation}

The new sets are unions of disjoint level-$i$ sets, and every level-$i$ set is used exactly
once. Hence they are pairwise disjoint and cover $\mathcal W_1$, which proves
\textup{(P1)}. Every word in a new set projects to $J_{i+1}$, as
$
  \tau_{i+1}(I)=\pi_{i+1}(\tau_i(I))=J_{i+1}
$
for every word $I$ in that set. This proves \textup{(P2)}.

To prove \textup{(P3)}, first take $1\le\ell\le i$ and
$L\in\tau_\ell(G_{i+1}(r_1,\ldots,r_i;J_{i+1}))$. Among the level-$i$ sets whose union
defines the new set in \eqref{eq:next-level-group}, exactly one contains a word whose
$\tau_\ell$-image is $L$. Indeed, suppose that two sets
$G_i(r_1,\ldots,r_{i-1};J_i)$ and $G_i(r_1,\ldots,r_{i-1};J_i')$ have
$\tau_\ell$-images containing $L$. If $\ell<i$, then
\[
  J_i=(\pi_i\circ\cdots\circ\pi_{\ell+1})(L)=J_i'.
\]
If $\ell=i$, then $L=J_i=J_i'$. Thus the two sets are the same. Therefore the fibre over
$L$ is contained in one level-$i$ set, and the
induction hypothesis gives
\[
  \#\{I\in G_{i+1}(r_1,\ldots,r_i;J_{i+1}):\tau_\ell(I)=L\}
  =2^{r_1+\cdots+r_{\ell-1}}.
\]
For $\ell=i+1$, the $\tau_\ell$-image of
$G_{i+1}(r_1,\ldots,r_i;J_{i+1})$ consists only of $J_{i+1}$. The union in
\eqref{eq:next-level-group} contains $2^{r_i}$ level-$i$ sets, and each has
$2^{r_1+\cdots+r_{i-1}}$ words by
\eqref{eq:inductive-group-size}. Hence
\[
  \#G_{i+1}(r_1,\ldots,r_i;J_{i+1})
  =2^{r_i}2^{r_1+\cdots+r_{i-1}}
  =2^{r_1+\cdots+r_i}.
\]
This proves \textup{(P3)} and completes the induction.
\smallskip

After the level-$s$ construction, fix a complete index vector
$\mathbf r=(r_1,\ldots,r_{s-1})$ and let
\begin{equation*}
  C_{\mathbf r}
  =\bigcup_{\substack{J_s\in\mathcal W_s\\
  G_s(r_1,\ldots,r_{s-1};J_s)\ne\varnothing}}
  G_s(r_1,\ldots,r_{s-1};J_s).
\end{equation*}
Retain only the nonempty sets $C_{\mathbf r}$. Since the level-$s$ sets form a partition of $\mathcal W_1$, the
sets $C_{\mathbf r}$ also form a partition of $\mathcal W_1$. Fix $1\le i\le s$ and
$J_i\in\tau_i(C_{\mathbf r})$.
For $1\le i<s$, the word $J_i$ has the uniquely determined image
$J_s=(\pi_s\circ\cdots\circ\pi_{i+1})(J_i)\in\mathcal W_s$; when $i=s$, set $J_s=J_i$.
Applying \eqref{eq:inductive-equal-fibres} at level $s$ gives
\begin{equation*}
  m_i(C_{\mathbf r},J_i)=2^{r_1+\cdots+r_{i-1}},\qquad 1\le i\le s.
\end{equation*}
The right-hand side is independent of $J_i$. Therefore $C_{\mathbf r}$ has uniform fibres
at every level, with $M_1(C_{\mathbf r})=1$ and
$M_i(C_{\mathbf r})=2^{r_1+\cdots+r_{i-1}}$ for $2\le i\le s$.
\smallskip

It remains to count the possible vectors $\mathbf r$. At the first step,
$q_1(J_2)$ is the number of words in a fibre of
$\pi_2:\mathcal W_1\to\mathcal W_2$, so
\[
  q_1(J_2)\le(\#\cS_1)^N.
\]
For $2\le i\le s-1$, each member of \eqref{eq:family-to-be-grouped} corresponds to a
different $J_i$ in the fibre of $\pi_{i+1}:\mathcal W_i\to\mathcal W_{i+1}$. Every word in
$\mathcal W_i$ has length $N$ over the alphabet $\cS_i$, and hence
\[
  q_i(r_1,\ldots,r_{i-1};J_{i+1})
  \le\#\pi_{i+1}^{-1}(J_{i+1})
  \le(\#\cS_i)^N.
\]
Thus, whenever an exponent $r_i$ occurs in one of the binary expansions,
$q_i(r_1,\ldots,r_{i-1};J_{i+1})$ is at least $2^{r_i}$ and at most $(\#\cS_i)^N$.
Consequently,
\[
  0\le r_i\le N\frac{\log\#\cS_i}{\log2}
  \le N\left\lceil\frac{\log\#\cS_i}{\log2}\right\rceil.
\]
Thus there are at most
\[
  \prod_{i=1}^{s-1}
  \left(1+N\left\lceil
    \frac{\log\#\cS_i}{\log2}\right\rceil\right)
\]
possible complete index vectors, which proves \eqref{eq:number-partition-parts}.
\end{proof}

\begin{remark}
Let $C\subset\mathcal W_1$ have uniform fibres at every level, and let $\PP_C$ be the uniform
probability distribution on $C$. Then \eqref{eq:partition-cardinality} gives that
$(\tau_i)_*\PP_C$ is the uniform probability distribution on $\tau_i(C)$, i.e.
\begin{equation}\label{eq:exact-factor-mass}
  (\tau_i)_*\PP_C(\{J\})=\frac{M_i(C)}{\#C}=\frac{1}{B_i(C)}
  \qquad J\in\tau_i(C).
\end{equation}
\end{remark}

\begin{lemma}\label{lem:partition-entropy}
Let $\zeta_N$ be a probability distribution on $\mathcal W_1$, and let
$\{C_t\}_{t=1}^{q_N}$ be the partition given by Lemma~\ref{lem:word-partition}. Write
$
  \vartheta_t=\zeta_N(C_t),
   B_{t,i}=B_i(C_t)=\#\tau_i(C_t)
$.
For $1\le i\le s$, let
$\zeta_{N,i}=\tau_i\zeta_N$ be the distribution on $\mathcal W_i$.
Then, for every $1\le i\le s$,
\begin{equation}\label{eq:partition-entropy}
  \sum_t\vartheta_t\log B_{t,i}
  \ge H(\zeta_{N,i})-\log q_N,
\end{equation}
where $
H(\zeta_{N,i})=-\sum_{J\in\mathcal W_i}\zeta_{N,i}(J)\log\zeta_{N,i}(J).
$
Moreover, there exist $m\in\{1,\ldots,s+1\}$, pairwise distinct indices $t_1,\ldots,t_m$,
and positive numbers $\rho_1,\ldots,\rho_m$ with $\sum_{j=1}^m\rho_j=1$ such that
\[
  \sum_t\vartheta_t\log B_{t,i}
  =\sum_{j=1}^m\rho_j\log B_{t_j,i}
  \qquad 1\le i\le s.
\]
\end{lemma}

\begin{proof}
For $\vartheta_t>0$, let $p_{t,i}$ be the probability distribution on $\mathcal W_i$ given by
\[
  p_{t,i}(J)
  =\frac{1}{\vartheta_t}
    \sum_{\substack{I\in C_t\\ \tau_i(I)=J}}\zeta_N(I).
\]
This distribution is supported on $\tau_i(C_t)$, which has $B_{t,i}$ elements, and hence
$H(p_{t,i})\le\log B_{t,i}$.
For $J\in\mathcal W_i$, write
$a_{t,J}=\vartheta_t p_{t,i}(J)$. Then
$\zeta_{N,i}(J)=\sum_{t:\vartheta_t>0}a_{t,J}$ and
$a_{t,J}\le\zeta_{N,i}(J)$. Therefore,
\begin{align*}
  H(\zeta_{N,i})
  &=-\sum_{t:\vartheta_t>0}\sum_{J\in\mathcal W_i}
       a_{t,J}\log\zeta_{N,i}(J)
  \le-\sum_{t:\vartheta_t>0}\sum_{J\in\mathcal W_i}a_{t,J}\log a_{t,J}\\
  &=-\sum_{t:\vartheta_t>0}\vartheta_t\log\vartheta_t
    +\sum_{t:\vartheta_t>0}\vartheta_tH(p_{t,i})
  \le\log q_N+\sum_t\vartheta_t\log B_{t,i}.
\end{align*}
This proves \eqref{eq:partition-entropy}.

For the last assertion, write
\[
  v_t=(\log B_{t,1},\ldots,\log B_{t,s})\in\RR^s.
\]
The vector
$
  v=\sum_t\vartheta_tv_t
$
belongs to the convex hull of the vectors $v_t$ in $\RR^s$. By Carath\'eodory's theorem,
every point in the convex hull of a subset of $\RR^s$ is a convex combination of at most
$s+1$ points of that subset.
There are $m\le s+1$ pairwise distinct indices $t_1,\ldots,t_m$ and coefficients
$\rho_1,\ldots,\rho_m>0$, with
$\sum_{j=1}^m\rho_j=1$, such that
\[
  v=\sum_{j=1}^m\rho_jv_{t_j}.
\]
Equality of the $i$-th coordinates gives the last assertion of the lemma.
\end{proof}

\begin{corollary}\label{cor:block-construction}
Suppose $\omega\in\cM_\sigma(X)$ satisfies, for every $N\ge1$,
\[
  H((\tau_i\omega)|\cL_N(X_i))=Nh(\tau_i\omega)+O(1)
  \qquad 1\le i\le s.
\]
For every $N\ge2$, let $\{C_t\}_{t=1}^{q_N}$ be a partition supplied by
Lemma~\ref{lem:word-partition}. Then there exist $m\in\{1,\ldots,s+1\}$, pairwise distinct
sets $C_{t_1},\ldots,C_{t_m}\subset\cL_N(X)$, and positive numbers
$\rho_1,\ldots,\rho_m$ with $\sum_{j=1}^m\rho_j=1$.
If
$B_{t_j,i}=\#\tau_i(C_{t_j})$, then
\begin{equation}\notag
  \sum_{j=1}^m\rho_j\log B_{t_j,i}
  \ge Nh(\tau_i\omega)-O(\log N)
  \qquad 1\le i\le s.
\end{equation}
\end{corollary}

\begin{proof}
Let $\zeta_N$ be the probability distribution on $\cL_N(X)$ given by
\[
  \zeta_N(I)=\omega([I])
  \qquad I\in\cL_N(X).
\]
For every $i$, write $\zeta_{N,i}=\tau_i\zeta_N$.
The hypothesis gives $H(\zeta_{N,i})=Nh(\tau_i\omega)+O(1)$ for every $i$.
Applying the two conclusions of Lemma~\ref{lem:partition-entropy} to $\zeta_N$, and using
$q_N=O(N^{s-1})$ from Lemma~\ref{lem:word-partition}, yields $m\le s+1$, pairwise distinct
indices $t_1,\ldots,t_m$, and positive coefficients $\rho_1,\ldots,\rho_m$ summing to $1$
such that, for every $1\le i\le s$,
\[
  \sum_{j=1}^m\rho_j\log B_{t_j,i}
  \ge H(\zeta_{N,i})-\log q_N
  =Nh(\tau_i\omega)-O(\log N).
\]
This proves the corollary.
\end{proof}

Corollary~\ref{cor:block-construction} gives finitely many admissible word sets together with
weights describing the proportions in which they should be used. In the measure construction in
Section~\ref{sec:measure-construction}, we need a sequence of these sets whose occurrence frequencies follow the prescribed
weights. The next lemma constructs such a sequence with a uniformly bounded
error.

\begin{lemma}\label{lem:discrepancy}
Let $q\in\NN$, let $\rho_1,\ldots,\rho_q>0$, and suppose
$\sum_{t=1}^q\rho_t=1$. There is a sequence
$u_1,u_2,\ldots\in\{1,\ldots,q\}$ with the following property.
For $r\in\NN$, let
$
  N_t(r)=\#\{1\le j\le r:u_j=t\}.
$
Then, for every
$1\le t\le q$ and $r\in\NN$,
\begin{equation}\label{eq:prefix-discrepancy}
  |N_t(r)-\rho_t r|\le q-1.
\end{equation}
Consequently, for every $1\le t\le q$ and all integers $0\le a<b$,
\[
  \left|\#\{a<j\le b:u_j=t\}-\rho_t(b-a)\right|\le2(q-1).
\]
\end{lemma}

\begin{proof}
If $q=1$, take $u_r=1$ for every $r$. Hence assume $q\ge2$.

Choose $u_1$ arbitrarily. For $1\le t\le q$ and $r\in\NN$, write
$
  e_t(r)=N_t(r)-r\rho_t.
$
Having chosen the first $r$ terms, where $r\ge1$, write
$
  p_t(r)=(r+1)\rho_t-N_t(r)=\rho_t-e_t(r),
$
and choose $u_{r+1}$ to be an index for which $p_t(r)$ is largest. Since
\[
  \sum_{t=1}^q p_t(r)
  =(r+1)\sum_{t=1}^q\rho_t-\sum_{t=1}^qN_t(r)
  =(r+1)-r=1,
\]
the largest number $p_t(r)$ is positive.

We prove by induction that for $1\le t\le q$ and $r\in\NN$,
\begin{equation}\label{eq:one-sided-discrepancy}
  e_t(r)<1.
\end{equation}
For $r=1$, this is immediate. Assume it holds at time $r$, and let $t_*=u_{r+1}$ be the chosen
index. For this index,
\[
  e_{t_*}(r+1)
  =e_{t_*}(r)+1-\rho_{t_*}
  =1-p_{t_*}(r)<1,
\]
since $p_{t_*}(r)>0$. If $t\ne t_*$, then
\[
  e_t(r+1)=e_t(r)-\rho_t<1,
\]
by the induction hypothesis and $\rho_t>0$. This proves
\eqref{eq:one-sided-discrepancy}.

For every $r\in\NN$ we also have
\[
  \sum_{t=1}^q e_t(r)
  =\sum_{t=1}^qN_t(r)-r\sum_{t=1}^q\rho_t=r-r=0.
\]
Consequently, for each $1\le t\le q$,
\[
  e_t(r)=-\sum_{\substack{1\le v\le q\\v\ne t}}e_v(r)>-(q-1).
\]
Together with \eqref{eq:one-sided-discrepancy}, this gives
$|e_t(r)|\le q-1$, and hence proves
\eqref{eq:prefix-discrepancy}.

Finally, for $1\le a<b$,
\[
  \#\{a<j\le b:u_j=t\}-\rho_t(b-a)
  =N_t(b)-N_t(a)-\rho_t(b-a)=e_t(b)-e_t(a).
\]
This gives the stated bound when $a\ge1$; when $a=0$, it follows directly from
\eqref{eq:prefix-discrepancy}.
\end{proof}

\section{Proof of Proposition~\ref{thm:pressure-to-measure}}\label{sec:measure-construction}

The preceding section produced finite sets of admissible words with uniform fibres and
factor cardinalities that give a required entropy lower bound. We now use these sets to choose
independent random source blocks, join the chosen blocks by weak specification, and push the
resulting product measure forward to $K$. The following
proposition gives the quantitative estimate needed for
Proposition~\ref{thm:pressure-to-measure}.

\begin{proposition}\label{thm:measure-construction}
Let $X$ be a subshift satisfying weak specification and set $K=R(X)$.
Suppose that $\eta,\xi\in\cM_\sigma(X)$ satisfy the following conditions.
\begin{enumerate}[label=\textup{(\roman*)}]
  \item for every $\omega\in\{\eta,\xi\}$, $1\le i\le s$, and $N\ge1$,
  \(
     H((\tau_i\omega)|\cL_N(X_i))=Nh(\tau_i\omega)+O_{\omega}(1);
  \)
  \item $\eta$ is an $\mathbf a$-weighted equilibrium state, i.e.,
  $\sum_{i=1}^s a_i h(\tau_i\eta)=P^{\mathbf a}(\sigma)$;
  \item there is $\beta\in(1/2,1)$ such that
  $\sum_{i=1}^s\Delta_i(\theta_i^\beta-\theta_{i-1}^\beta)>0$, where
  $\Delta_i=h(\tau_i\xi)-h(\tau_i\eta)$ for $1\le i\le s$.
\end{enumerate}
Then there exist a Borel probability measure $\lambda$ supported on $K$, a constant $c_0>0$, and an
integer $k_0$ such that
\begin{equation}\label{eq:cube-mass-decay}
  \log\lambda(Q_k(z))+kP^{\mathbf a}(\sigma)\le-c_0k^\beta
\end{equation}
for every $z\in K$ and every $k\ge k_0$.
\end{proposition}

We first derive Proposition~\ref{thm:pressure-to-measure} from
Proposition~\ref{thm:measure-construction}.

\begin{proof}[Proof of Proposition~\ref{thm:pressure-to-measure} based on Proposition~\ref{thm:measure-construction}]
Let $\eta=\eta_{\mathbf a}$ and $\xi=\eta_{\mathbf a(\beta)}$. Then
$P^{\mathbf a(\beta)}(\sigma)=\sum_{i=1}^s a_i(\beta)h(\tau_i\xi)$.
Proposition~\ref{prop:known-results}\textup{(iii)} shows that
$\eta$ and $\xi$ satisfy condition \textup{(i)} of
Proposition~\ref{thm:measure-construction}. Since $\eta=\eta_{\mathbf a}$,
Proposition~\ref{prop:known-results}\textup{(ii)} gives condition \textup{(ii)}. With $\Delta_i$
as in condition \textup{(iii)}, the definition of $c(\beta)$ in \eqref{eq:pressure-gap} gives
\begin{align*}
  \sum_{i=1}^s\Delta_i
  (\theta_i^\beta-\theta_{i-1}^\beta)
  =\sum_{i=1}^sa_i(\beta)
    \bigl(h(\tau_i\xi)-h(\tau_i\eta)\bigr)
  =c(\beta)>0.
\end{align*}
Proposition~\ref{thm:measure-construction} gives a Borel probability measure
$\lambda$ on $K$ and constants $c_0>0$ and $k_0$ such that, for every $z\in K$ and
$k\ge k_0$,
\begin{equation}\label{eq:pressure-mass-bound}
  \lambda(Q_k(z))
  \le\exp\{-kP^{\mathbf a}(\sigma)-c_0k^\beta\}.
\end{equation}

Write $L=\log n_s$ and $\rho_k=n_s^{-k}=e^{-kL}$ for $k\ge1$. By
\eqref{eq:dimension-formulas}, $P^{\mathbf a}(\sigma)=\gamma\log n_s=\gamma L$.
Choose $0<\varepsilon<c_0/L^\beta$ and write
$\delta=c_0-\varepsilon L^\beta>0$.
With this value of $\varepsilon$, let $g_\beta$ be a gauge function satisfying
$g_\beta(r)=r^\gamma\exp\{-\varepsilon(\log(1/r))^\beta\}$ for $0<r\le1$.
Then
\begin{equation}\label{eq:gauge-at-grid-scale}
  g_\beta(\rho_k)=\exp\{-kP^{\mathbf a}(\sigma)-\varepsilon L^\beta k^\beta\}.
\end{equation}
For every $z\in K$, equations \eqref{eq:pressure-mass-bound} and
\eqref{eq:gauge-at-grid-scale} give
\begin{equation}\label{eq:grid-zero-density}
  0\le
  \frac{\lambda(Q_k(z))}{g_\beta(\rho_k)}
  \le e^{-\delta k^\beta}
  \longrightarrow0\qquad k\to\infty.
\end{equation}

Lemma~\ref{lem:density}, applied to
\eqref{eq:grid-zero-density}, shows that $\HH^{g_\beta}(K)=\infty$ and that
$\HH^{g_\beta}$ is not $\sigma$-finite on $K$. Since $g_\beta(r)\le r^\gamma$, we have
$\HH^{g_\beta}(E)\le\HH^\gamma(E)$ for every $E\subset K$. Therefore
$\HH^\gamma(K)=\infty$, and $\HH^\gamma$ is not $\sigma$-finite on $K$.
\end{proof}

The proof of Proposition~\ref{thm:measure-construction} has four parts. We first choose the
source-block coordinate sets $C_r$ in Lemma~\ref{lem:word-partition}, use
Lemma~\ref{lem:discrepancy} to schedule them, and construct the product measure $\PP$. We then
use weak specification to insert joining words and define the push-forward measure $\lambda$.
Next, we estimate the $\lambda$-mass of approximate cubes. Finally, the entropy estimates from
Corollary~\ref{cor:block-construction}, together with the block schedule, prove
Proposition~\ref{thm:measure-construction}.

\subsection{Measure construction}

We shall construct a product space whose coordinates are admissible words, called
\emph{source blocks}. We first prescribe the length of every source block and then choose
the finite set of words allowed at each coordinate. Fix a value $\beta\in(1/2,1)$ satisfying condition
\textup{(iii)} of Proposition~\ref{thm:measure-construction}.

For $j\ge1$, let $R_j=L_j=2^j$ and $Q_j=\sum_{h=1}^jR_h=2^{j+1}-2$, with $Q_0=0$.
For $j\ge1$, the source blocks with indices $Q_{j-1}<r\le Q_j$ form stage $j$. For every
such $r$, let the $r$-th source block have length $L(r)=L_j$.
For $r\ge1$, place the first $r$ source blocks consecutively and denote their right endpoint
by $\overline s_r=\sum_{u=1}^rL(u)$; set $\overline s_0=0$.
For $j\ge1$, denote the right endpoint of stage $j$ in these source coordinates by
$\overline S_j=\overline s_{Q_j}=\sum_{h=1}^jR_hL_h$; set $\overline S_0=0$.
For $n\in\NN\cup\{0\}$, let $f(n)=(n+1)^{\beta-1}$, so that $0<f(n)\le1$.

\begin{lemma}\label{lem:source-product-measure}

Under the assumptions of Proposition~\ref{thm:measure-construction}, there exist
a sequence $(\varepsilon_r)_{r\ge1}\in\{0,1\}^{\NN}$ and finite nonempty sets
\[
  C_r\subset\cL_{L(r)}(X),\qquad r\ge1,
\]
each having uniform fibres at every level. Set
\[
  \Omega=\prod_{r=1}^{\infty}C_r.
\]
If $\nu_r$ is the uniform probability measure on $C_r$, then
$\PP=\bigotimes_{r=1}^{\infty}\nu_r$ is a Borel probability measure on $\Omega$.

In particular, for every $m\ge1$ and every choice
$J_r\in C_r$, $1\le r\le m$,
\begin{equation}\label{eq:source-product-cylinders}
  \PP\left\{(I_r)_{r\ge1}\in\Omega:I_1=J_1,\ldots,I_m=J_m \right\}
  =\prod_{r=1}^m\frac1{\#C_r}.
\end{equation}
The set $C_r$ is selected from the block construction for $\eta$ when
$\varepsilon_r=0$, and from the block construction for $\xi$ when $\varepsilon_r=1$.
\end{lemma}

The proof is straightforward. The sequence $(\varepsilon_r)_{r\ge1}$ and the sets $C_r$ will be
chosen in a particular way for later estimates.

\begin{proof}[Choice of $(\varepsilon_r)_{r\ge1}$ in Lemma~\ref{lem:source-product-measure}]
\renewcommand{\qedsymbol}{}
For $r\ge1$, let the interval of the $r$-th source block be
$\overline{\mathcal I}_r=(\overline s_{r-1},\overline s_r]\cap\NN$:
\[x_r=\frac 1{L(r)}\sum_{n\in\overline{\mathcal I}_r}f(n)\in[0,1].\]
For each $j\ge1$ and $1\le u\le R_j$, write
\begin{equation}\label{eq:stage-block-index}
r(j,u)=Q_{j-1}+u,
\qquad
A_{j,u}=\sum_{v=1}^u x_{r(j,v)},
\end{equation}
and set $A_{j,0}=0$. Since
$A_{j,u}-A_{j,u-1}=x_{r(j,u)}\in[0,1]$, we choose
\begin{equation}\label{eq:binary-schedule}
	\varepsilon_{r(j,u)}
	=\lfloor A_{j,u}\rfloor-\lfloor A_{j,u-1}\rfloor
	\in\{0,1\}.
\end{equation}
We assign the $r$-th source block to $\eta$ when $\varepsilon_r=0$ and to $\xi$
when $\varepsilon_r=1$.

\medskip
\noindent\emph{Choice of $(C_r)_{r\ge1}$ in Lemma~\ref{lem:source-product-measure}.} Condition
\textup{(i)} of Proposition~\ref{thm:measure-construction} allows us to apply
Corollary~\ref{cor:block-construction} to both $\eta$ and $\xi$. Fix $j$ and
$\omega\in\{\eta,\xi\}$. Applying Corollary~\ref{cor:block-construction} with $N=L_j$ gives
$m_{j,\omega}\le s+1$, sets
$C_{j,\omega,1},\ldots,C_{j,\omega,m_{j,\omega}}$, and weights
$\rho_{j,\omega,1},\ldots,\rho_{j,\omega,m_{j,\omega}}$. Let $N_{j,\omega}$ be the
number of source blocks in stage $j$ assigned to $\omega$. When $N_{j,\omega}>0$, list them as
$r_{j,\omega}(1)<\cdots<r_{j,\omega}(N_{j,\omega})$.
Lemma~\ref{lem:discrepancy} gives indices
$u_{j,\omega}(v)\in\{1,\ldots,m_{j,\omega}\}$, $1\le v\le N_{j,\omega}$, such that, for
every $1\le q\le m_{j,\omega}$ and $0\le a<b\le N_{j,\omega}$,
\begin{equation}\label{eq:type-discrepancy}
  \left|\#\{a<v\le b:u_{j,\omega}(v)=q\}
       -\rho_{j,\omega,q}(b-a)\right|
  \le2s.                                                    
\end{equation}
For $1\le v\le N_{j,\omega}$, we choose the $r_{j,\omega}(v)$-th source block set to be
\[
  C_{r_{j,\omega}(v)}=C_{j,\omega,u_{j,\omega}(v)}.
\]
\end{proof}
For $t\in\NN$, write $E_\beta(t)=\sum_{n=1}^t f(n)$, and set $E_\beta(0)=0$.
To describe the source blocks before the joining words are inserted, define, for every
integer $t\ge1$,
\begin{equation*}
  r_*(t)=\min\{r\ge1:t\le\overline s_r\},
  \qquad
  L_*(t)=L(r_*(t)).
\end{equation*}
Thus $r_*(t)$ and $L_*(t)$ give, respectively, the number and the maximum length of the
source blocks that meet the first $t$ digit positions before the joining words are inserted.

We now introduce the joining words. For each source-block sequence
$\mathbf I=(I_r)_{r\ge1}\in\Omega$, we shall insert joining words and obtain
\begin{equation}\label{eq:source-block-concatenation}
  I_1W_1(\mathbf I)I_2W_2(\mathbf I)I_3W_3(\mathbf I)\cdots \in X,
\end{equation}
Here $W_r(\mathbf I)$ is the \emph{joining word} placed between $I_r$ and
$I_{r+1}$, and $|W_r(\mathbf I)|\le p$ by weak specification. The
sequence in \eqref{eq:source-block-concatenation} is called the \emph{concatenated
sequence}. Let $S_j(\mathbf I)$ be the position of the last digit of the final source
block in stage $j$, including all joining words preceding that block, and let $\mathcal I_r(\mathbf I)\subset\NN$ be the
interval occupied by $I_r$ in the concatenated sequence in
\eqref{eq:source-block-concatenation}. For every integer $t\ge1$, let
\begin{equation*}
  M_{\mathbf I}(t)=\sum_{r\ge1}\varepsilon_r
       \#\bigl(\mathcal I_r(\mathbf I)\cap[1,t]\bigr).
\end{equation*}
Set $M_{\mathbf I}(0)=0$.
Thus $M_{\mathbf I}(t)$ counts the digits belonging to $\xi$-blocks among the first
$t$ positions of the concatenated sequence in \eqref{eq:source-block-concatenation}. It does
not count digits in joining words.

\begin{lemma}\label{lem:scheduled-measure}
Let $(\varepsilon_r)_{r\ge1}$, $(C_r)_{r\ge1}$, $\Omega$, and $\PP$ be given by
Lemma~\ref{lem:source-product-measure}. There is a continuous map $\Pi:\Omega\to X$.
Define the push-forward measure by
\begin{equation}\label{eq:lambda-definition}
  \lambda=(R\circ\Pi)_*\PP.
\end{equation}
Then $\lambda$ is a Borel probability measure on $K$.
The following estimates hold for $t\ge1$.
\begin{equation}\label{eq:block-scales}
  S_j(\mathbf I)\asymp4^j,
  \qquad r_*(t)=O(t^{1/2}),
  \qquad L_*(t)=O(t^{1/2}).
\end{equation}
Moreover, for every $\mathbf I\in\Omega$ and $t\ge1$,
\begin{equation}\label{eq:M-tracking}
  \left|M_{\mathbf I}(t)-E_\beta(t)\right|
  \lesssim L_*(t)+p\,r_*(t).
\end{equation}
\end{lemma}

\begin{proof}
We first construct the map $\Pi$ and the probability measure $\lambda$ on $K$.
Let $\cL_{\le p}(X)$ be the finite set of admissible words of length at most $p$. Weak specification implies that, for every $U,V\in\cL(X)$, the set
\[
  \{W\in\cL_{\le p}(X):UWV\in\cL(X)\}
\]
is nonempty. For each pair $(U,V)$, choose one word from this set and denote it by
$\mathscr W(U,V)$. This choice is fixed throughout the construction.

Fix $\mathbf I=(I_r)_{r\ge1}\in\Omega$. For $r\ge1$, denote
\begin{equation}\label{eq:joining-word-choice}
  W_r(\mathbf I)=\mathscr W(I_r,I_{r+1}).
\end{equation}
Every finite
prefix of the following infinite word is admissible, and therefore
\[
  \Pi(\mathbf I)
  =I_1W_1(\mathbf I)I_2W_2(\mathbf I)I_3W_3(\mathbf I)\cdots\in X.
\]
For every $n$, the first $n$ digits of $\Pi(\mathbf I)$ depend on only finitely many
coordinates $I_r$ of $\mathbf I$. Hence the inverse image under $\Pi$ of every cylinder in
$X$ is open in $\Omega$. Thus $\Pi$ is continuous and, in particular, Borel measurable.
The composition $R\circ\Pi:\Omega\to K$ is Borel measurable. With $\lambda$ given by
\eqref{eq:lambda-definition}, $\lambda$ is therefore a Borel probability measure on $K$.
\smallskip

For the estimates, fix $\mathbf I\in\Omega$ and abbreviate
$W_r=W_r(\mathbf I)$, $S_j=S_j(\mathbf I)$,
$\mathcal I_r=\mathcal I_r(\mathbf I)$, and $M(t)=M_{\mathbf I}(t)$.
Since $R_h=L_h=2^h$,
\begin{equation*}
  \overline S_j=\sum_{h=1}^jR_hL_h=\sum_{h=1}^j4^h\asymp4^j.
\end{equation*}
For $t\ge1$, choose $j\ge1$ so that
\(\overline S_{j-1}\le t<\overline S_j\). Then
\[
  r_*(t)\le Q_j\le2^{j+1},
  \qquad
  L_*(t)\le L_j=2^j.
\]
If $j\ge2$, then $t\ge\overline S_{j-1}\ge4^{j-1}$, and hence
\begin{equation}\label{eq:block-index-scales}
  r_*(t)=O(t^{1/2}),
  \qquad
  L_*(t)=O(t^{1/2}).
\end{equation}
By its definition,
\[
  S_j=\overline S_j+\sum_{r=1}^{Q_j-1}|W_r|.
\]
Consequently,
\begin{equation}\label{eq:joined-stage-scale}
  \overline S_j\le S_j\le\overline S_j+pQ_j,
  \qquad
  S_j\asymp4^j.
\end{equation}
Equations \eqref{eq:block-index-scales} and \eqref{eq:joined-stage-scale} prove
\eqref{eq:block-scales}. 
\smallskip 

Inserting joining words only moves a source block to the right.
Therefore, if $I_r$ meets one of the first $t$ positions of the concatenated sequence, then $r\le r_*(t)$. By the
definition of $\mathcal I_r$,
\begin{equation*}
  0\le\min\mathcal I_r-(\overline s_{r-1}+1)
  =\sum_{u<r}|W_u|\le p(r-1).
\end{equation*}
For $t\in\NN$, let $\overline M(t)$ be the number of digits
belonging to $\xi$-blocks among the first $t$ positions before the joining words are
inserted, and set $\overline M(0)=0$. For $t\ge1$, choose the unique integers
$j\ge1$, $0\le u<R_j$, and $0\le q<L_j$ such that
\[
  t=\overline S_{j-1}+uL_j+q.
\]
With
$T=\overline S_{j-1}+uL_j$, and using $r(h,v)=Q_{h-1}+v$ from
\eqref{eq:stage-block-index}, we have
\begin{align*}
  \overline M(t)
  & =\sum_{h<j}L_h\sum_{v=1}^{R_h}\varepsilon_{r(h,v)}
     +L_j\sum_{v=1}^u\varepsilon_{r(j,v)}
     +q\varepsilon_{r(j,u+1)},\\
  E_\beta(t)
  & =\sum_{h<j}L_h\sum_{v=1}^{R_h}x_{r(h,v)}
     +L_j\sum_{v=1}^ux_{r(j,v)}
     +\sum_{n=T+1}^{T+q}f(n).
\end{align*}
For every $h\ge1$ and $0\le w\le R_h$, summing \eqref{eq:binary-schedule} gives
\[
  \sum_{v=1}^w\varepsilon_{r(h,v)}=\lfloor A_{h,w}\rfloor,
  \qquad
  \left|
    \sum_{v=1}^w\varepsilon_{r(h,v)}
    -\sum_{v=1}^w x_{r(h,v)}
  \right|
  =\bigl|\lfloor A_{h,w}\rfloor-A_{h,w}\bigr|<1.
\]
Apply this estimate with $w=R_h$ for every $h<j$, and then with $(h,w)=(j,u)$.
Since $0\le f\le1$,
\begin{equation}\label{eq:source-time-tracking}
  \left|\overline M(t)-E_\beta(t)\right|
  \le\sum_{h<j}L_h+L_j+q\le\sum_{h<j}L_h+2L_j\lesssim L_j\lesssim L_*(t).
\end{equation}

For an integer $t\ge1$, let $V(t)$ be the number of joining-word digits among the first
$t$ positions of the concatenated sequence, and
write $U(t)=t-V(t)$. Since the first position belongs to $I_1$, we have $U(t)\ge1$. By the
definitions of $M(t)$, $\overline M(t)$, and $U(t)$,
\[
  M(t)=\overline M(U(t)),
  \qquad
  1\le U(t)\le t.
\]
Moreover,
\[
  V(t)\le\sum_{r=1}^{r_*(U(t))}|W_r|\le p\,r_*(U(t)).
\]
Since $0\le f(n)\le1$,
\[
  \left|E_\beta(t)-E_\beta(U(t))\right|
  =\sum_{n=U(t)+1}^{t}f(n)
  \le t-U(t)=V(t).
\]
Applying \eqref{eq:source-time-tracking} and the monotonicity of $L_*$ and $r_*$ gives
\begin{align*}
  \left|M(t)-E_\beta(t)\right|
  &\le\left|\overline M(U(t))-E_\beta(U(t))\right|
      +\left|E_\beta(U(t))-E_\beta(t)\right|\\
  &\lesssim L_*(U(t))+p\,r_*(U(t))\\
  &\lesssim L_*(t)+p\,r_*(t),
\end{align*}
which proves \eqref{eq:M-tracking} uniformly for $\mathbf I\in\Omega$.
\end{proof}

\subsection{The mass for approximate cubes}

Lemma~\ref{lem:source-product-measure} makes the source blocks independent, whereas each
joining word inserted in Lemma~\ref{lem:scheduled-measure} depends on the two source blocks it
joins. Their lengths, however, take only $p+1$ possible values. We therefore first fix the
entire joining-length pattern and then drop the requirement that the joining words have the
prescribed lengths, which can only enlarge the event. The remaining conditions concern distinct
source blocks and, by their independence under the original product measure, the probability of
the enlarged event is the product of the corresponding probabilities.

For the sets $C_r$ constructed in Lemma~\ref{lem:source-product-measure}, write
$B_{r,i}=\#\tau_i(C_r)$ for $r\ge1$ and $1\le i\le s$.

\begin{lemma}
Let $\lambda$ be the probability measure constructed in
Lemma~\ref{lem:scheduled-measure} from the sets $C_r$ given by
Lemma~\ref{lem:source-product-measure}. Fix an integer $k\ge1$. For $1\le i\le s$, let
$
  I_i(k)=(\lfloor\theta_{i-1}k\rfloor,\lfloor\theta_i k\rfloor]\cap\NN.
$
Write also $r_k=r_*(k)$.
Let
\[
  \mathfrak L_k
  =\left\{
    (|W_1(\mathbf I)|,\ldots,|W_{r_k}(\mathbf I)|):
    \mathbf I\in\Omega
   \right\},
\]
where $W_r(\mathbf{I})$ is defined in \eqref{eq:joining-word-choice}.
For $\boldsymbol\ell=(\ell_1,\ldots,\ell_{r_k})\in\mathfrak L_k$ and
$1\le r\le r_k$, define the left and right endpoints of the $r$-th source block by
\begin{equation}\label{eq:length-pattern-block-position}
  a_r(\boldsymbol\ell)=\overline s_{r-1}+\sum_{u=1}^{r-1}\ell_u,
  \qquad
  b_r(\boldsymbol\ell)=a_r(\boldsymbol\ell)+L(r),
\end{equation}
and let
$\mathcal I_r(\boldsymbol\ell)
=(a_r(\boldsymbol\ell),b_r(\boldsymbol\ell)]\cap\NN$.
For $1\le i\le s$, let
$
  \mathcal B_{i,k}(\boldsymbol\ell)
  =\{1\le r\le r_k:\mathcal I_r(\boldsymbol\ell)\subset I_i(k)\}.
$
Then
\begin{equation}\label{eq:number-length-patterns}
  \#\mathfrak L_k\le(p+1)^{r_k}.
\end{equation}
Moreover, for every $z\in K$,
\begin{equation}\label{eq:geometric-product}
  \lambda(Q_k(z))
  \le2^d(p+1)^{r_k}
  \max_{\boldsymbol\ell\in\mathfrak L_k}
  \prod_{i=1}^s
  \prod_{r\in\mathcal B_{i,k}(\boldsymbol\ell)}B_{r,i}^{-1}.
\end{equation}
\end{lemma}

\begin{proof}

Fix $\mathbf I\in\Omega$. If the source block $I_r$ or the joining word $W_r(\mathbf I)$
meets the first $k$ positions of the concatenated sequence, then $r\le r_k$. Consequently,
every source block contained in one of the intervals $I_i(k)$ has index at most $r_k$, and
its location is determined by
$|W_1(\mathbf I)|,\ldots,|W_{r_k}(\mathbf I)|$. Since this holds for every
$\mathbf I\in\Omega$ and every joining-word length belongs to $\{0,\ldots,p\}$, there are at
most $(p+1)^{r_k}$ possible length patterns. This proves
\eqref{eq:number-length-patterns}.

\smallskip

For $\boldsymbol\ell\in\mathfrak L_k$, let
\[
  \Omega_{\boldsymbol\ell}
  =\left\{
    \mathbf I=(I_u)_{u\ge1}\in\Omega:
    (|W_1(\mathbf I)|,\ldots,|W_{r_k}(\mathbf I)|)
    =\boldsymbol\ell
   \right\}.
\]
The word $W_r(\mathbf I)$ depends only on $I_1,\ldots,I_{r+1}$, so every
$\Omega_{\boldsymbol\ell}$ is a Borel subset of $\Omega$.
These sets are pairwise disjoint and their union is the whole product probability space:
\begin{equation}\label{eq:length-event-partition}
  \Omega=\bigsqcup_{\boldsymbol\ell\in\mathfrak L_k}
  \Omega_{\boldsymbol\ell}.
\end{equation}
By \eqref{eq:length-pattern-block-position}, for every
$\boldsymbol\ell\in\mathfrak L_k$, $\mathbf I\in\Omega_{\boldsymbol\ell}$, and
$1\le r\le r_k$, the $r$-th source block occupies the interval
$\mathcal I_r(\boldsymbol\ell)$ in the concatenated sequence.
For $r\ge1$, let $\operatorname{pr}_r:\Omega\to C_r$ be the $r$-th coordinate projection,
so that $\operatorname{pr}_r((I_u)_{u\ge1})=I_r$.
Fix $x\in X$ and $\boldsymbol\ell\in\mathfrak L_k$.
For $1\le i\le s$ and $r\in\mathcal B_{i,k}(\boldsymbol\ell)$, let
\[
  A_{r,i}(x,\boldsymbol\ell)
  =\left\{
    \mathbf I=(I_u)_{u\ge1}\in\Omega:
    \tau_i(\operatorname{pr}_r(\mathbf I))
    =\tau_i\bigl(x|_{a_r(\boldsymbol\ell)}^{b_r(\boldsymbol\ell)}\bigr)
   \right\}.
\]
Thus $A_{r,i}(x,\boldsymbol\ell)$ is the set of source-block sequences for which the
$i$-th factor of the $r$-th source block agrees with the corresponding factor word of $x$.
The definition of $\Gamma_k(x)$ in \eqref{eq:symbolic-cube} gives
\begin{equation*}
  \Pi^{-1}([\Gamma_k(x)])\cap\Omega_{\boldsymbol\ell}
  \subseteq
  \bigcap_{i=1}^s\ \bigcap_{r\in\mathcal B_{i,k}(\boldsymbol\ell)}
  A_{r,i}(x,\boldsymbol\ell).
\end{equation*}
Indeed, let $\mathbf I$ belong to the set on the left, and fix $1\le i\le s$ and
$r\in\mathcal B_{i,k}(\boldsymbol\ell)$. Since $\mathbf I\in\Omega_{\boldsymbol\ell}$, the
$r$-th source block occupies the interval
$(a_r(\boldsymbol\ell),b_r(\boldsymbol\ell)]$. This interval is contained in $I_i(k)$ by the
definition of $\mathcal B_{i,k}(\boldsymbol\ell)$. On the other hand, since
$\Pi(\mathbf I)\in[\Gamma_k(x)]$, we have
\[
  \tau_i\bigl(\Pi(\mathbf I)|_{\lfloor\theta_i k\rfloor}\bigr)
  =\tau_i\bigl(x|_{\lfloor\theta_i k\rfloor}\bigr).
\]
Restricting this equality to $(a_r(\boldsymbol\ell),b_r(\boldsymbol\ell)]$ yields
\[
  \tau_i(I_r)
  =\tau_i\bigl(x|_{a_r(\boldsymbol\ell)}^{b_r(\boldsymbol\ell)}\bigr).
\]
Thus $\mathbf I\in A_{r,i}(x,\boldsymbol\ell)$, which proves the inclusion.

The intervals $I_i(k)$ are disjoint, so each block index $r$ belongs to at most one
set $\mathcal B_{i,k}(\boldsymbol\ell)$. Consequently, the sets
$A_{r,i}(x,\boldsymbol\ell)$ in the intersection are determined by pairwise distinct
coordinates of $\Omega$, and, by \eqref{eq:source-product-cylinders}, their probabilities
multiply under $\PP$. The $r$-th marginal
of $\PP$ is uniform on $C_r$, so \eqref{eq:exact-factor-mass} gives

\[
  \PP(A_{r,i}(x,\boldsymbol\ell))
  \le B_{r,i}^{-1};
\]
the left-hand side is $0$ if the required factor word is not in $\tau_i(C_r)$, and otherwise
it equals $B_{r,i}^{-1}$. Hence
\begin{align*}
  \PP\!\left(
    \Pi^{-1}([\Gamma_k(x)])\cap\Omega_{\boldsymbol\ell}
  \right)
  &\le
  \PP\!\left(
    \bigcap_{i=1}^s\ \bigcap_{r\in\mathcal B_{i,k}(\boldsymbol\ell)}
    A_{r,i}(x,\boldsymbol\ell)
  \right)\notag\\
  &=\prod_{i=1}^s
    \prod_{r\in\mathcal B_{i,k}(\boldsymbol\ell)}
    \PP(A_{r,i}(x,\boldsymbol\ell))\notag\\
  &\le\prod_{i=1}^s
    \prod_{r\in\mathcal B_{i,k}(\boldsymbol\ell)}B_{r,i}^{-1}.
\end{align*}
Summing over the partition in \eqref{eq:length-event-partition} and using
\eqref{eq:number-length-patterns} gives
\begin{align}
  \PP\bigl(\Pi^{-1}([\Gamma_k(x)])\bigr)
  &=\sum_{\boldsymbol\ell\in\mathfrak L_k}
    \PP\!\left(
      \Pi^{-1}([\Gamma_k(x)])\cap\Omega_{\boldsymbol\ell}
    \right)\notag\\
  &\le\sum_{\boldsymbol\ell\in\mathfrak L_k}
    \prod_{i=1}^s
    \prod_{r\in\mathcal B_{i,k}(\boldsymbol\ell)}B_{r,i}^{-1}\notag\\
  &\le(p+1)^{r_k}
    \max_{\boldsymbol\ell\in\mathfrak L_k}
    \prod_{i=1}^s
    \prod_{r\in\mathcal B_{i,k}(\boldsymbol\ell)}B_{r,i}^{-1}.
                                                        \label{eq:symbolic-pattern-sum}
\end{align}

Fix $z\in K$ and write $Q=Q_k(z)$. By \eqref{eq:encoding-interface}, there are
$x^{(1)},\ldots,x^{(L)}\in X$, with $L\le2^d$, such that
\[
  R^{-1}(Q)\cap X
  \subseteq\bigcup_{h=1}^L[\Gamma_k(x^{(h)})].
\]
By \eqref{eq:lambda-definition},
$\lambda(Q)=\PP((R\circ\Pi)^{-1}(Q))$. Therefore
\[
  \lambda(Q)
  \le\sum_{h=1}^L
  \PP\!\left(\Pi^{-1}([\Gamma_k(x^{(h)})])\right).
\]
Applying \eqref{eq:symbolic-pattern-sum} to each $x^{(h)}$ and using $L\le2^d$ proves
\eqref{eq:geometric-product}.
\end{proof}

We now turn the product estimate in \eqref{eq:geometric-product} into an entropy lower bound.
Applying Corollary~\ref{cor:block-construction} block by block yields the estimate in the
following lemma.

\begin{lemma}
For the measure $\lambda$ constructed in Lemma~\ref{lem:scheduled-measure}, the following
estimate holds uniformly for $z\in K$ as $k\to\infty$:
\begin{align}
  -\log\lambda(Q_k(z))
  &\ge \sum_{i=1}^s(\lfloor\theta_i k\rfloor-\lfloor\theta_{i-1}k\rfloor)
       h(\tau_i\eta)\notag\\
  &\quad+\sum_{i=1}^s\Delta_i
  \bigl(E_\beta(\lfloor\theta_i k\rfloor)-E_\beta(\lfloor\theta_{i-1}k\rfloor)\bigr)
  -o(k^\beta).                                         \label{eq:master-information}
\end{align}
Here $\Delta_i=h(\tau_i\xi)-h(\tau_i\eta)$.
\end{lemma}

\begin{proof}
Fix $k\ge2$. Write $k_i=\lfloor\theta_i k\rfloor$ for $0\le i\le s$. For
$1\le i\le s$, write $D_i=k_i-k_{i-1}$ and
$I_i(k)=(k_{i-1},k_i]\cap\NN$. Let $r_k=r_*(k)$ and set
\begin{equation*}
  \mathcal E_k
  =r_k\log k+L_*(k)+p\,r_k+r_k\log(p+1).
\end{equation*}
By \eqref{eq:block-scales} and $\beta>1/2$,
\begin{equation}
  \frac{\mathcal E_k}{k^\beta}
  =O\bigl(k^{1/2-\beta}\log k\bigr)
  \longrightarrow0\qquad k\to\infty.                   \label{eq:total-error-small}
\end{equation}

Fix $1\le i\le s$, a stage $j\ge1$, and a measure $\omega\in\{\eta,\xi\}$. Let
$\mathcal J$ be the set of block indices in a consecutive segment of the sequence
$(r_{j,\omega}(v))_{1\le v\le N_{j,\omega}}$.
Thus, for some $0\le a<b\le N_{j,\omega}$,
$\mathcal J=\{r_{j,\omega}(v):a<v\le b\}$. Let $n=b-a=\#\mathcal J$ and, for
$1\le q\le m_{j,\omega}$, write
\[
  n_q=\#\{a<v\le b:u_{j,\omega}(v)=q\},
  \qquad
  B_{j,\omega,q,i}=\#\tau_i(C_{j,\omega,q}).
\]
Recall that $\cS_i=\tau_i(\cS)$ is the set of digits that occur in $X_i$.
Corollary~\ref{cor:block-construction}
and \eqref{eq:type-discrepancy} give
\begin{align*}
  \sum_{q=1}^{m_{j,\omega}}
     \rho_{j,\omega,q}\log B_{j,\omega,q,i}
  &\ge L_jh(\tau_i\omega)-O(\log L_j),\\
  |n_q-\rho_{j,\omega,q}n|&\le2s,\\
  \log B_{j,\omega,q,i}&\le L_j\log\#\cS_i.
\end{align*}
Therefore
\begin{align}
  \sum_{r\in\mathcal J}\log B_{r,i}
  &=\sum_{q=1}^{m_{j,\omega}}n_q\log B_{j,\omega,q,i}\notag\\
  &=n\sum_{q=1}^{m_{j,\omega}}
       \rho_{j,\omega,q}\log B_{j,\omega,q,i}
    +\sum_{q=1}^{m_{j,\omega}}
       (n_q-\rho_{j,\omega,q}n)\log B_{j,\omega,q,i}\notag\\
  &\ge nL_jh(\tau_i\omega)
    -O\bigl(n\log L_j\bigr)
    -2s\sum_{q=1}^{m_{j,\omega}}\log B_{j,\omega,q,i}\notag\\
  &\ge nL_jh(\tau_i\omega)
    -O\bigl(n\log L_j+L_j\bigr).         \label{eq:stage-information}
\end{align}

We now return to all factor levels. Fix $\boldsymbol\ell\in\mathfrak L_k$. For each
$1\le i\le s$, write
\[
  \ell_{\eta,i,k}(\boldsymbol\ell)
  =\sum_{\substack{r\in\mathcal B_{i,k}(\boldsymbol\ell)\\
                    \varepsilon_r=0}}L(r),
  \qquad
  \ell_{\xi,i,k}(\boldsymbol\ell)
  =\sum_{\substack{r\in\mathcal B_{i,k}(\boldsymbol\ell)\\
                    \varepsilon_r=1}}L(r).
\]
For $j\ge1$, $1\le i\le s$, and $\omega\in\{\eta,\xi\}$, let
\[
  \mathcal J_{j,\omega,i}(\boldsymbol\ell)
  =\begin{cases}
    \{r\in\mathcal B_{i,k}(\boldsymbol\ell):
      Q_{j-1}<r\le Q_j,\ \varepsilon_r=0\},&\omega=\eta,\\
    \{r\in\mathcal B_{i,k}(\boldsymbol\ell):
      Q_{j-1}<r\le Q_j,\ \varepsilon_r=1\},&\omega=\xi.
  \end{cases}
\]
If this set is nonempty, there are integers $a<b$ such that
\[
  \mathcal J_{j,\omega,i}(\boldsymbol\ell)
  =\{r_{j,\omega}(v):a<v\le b\}.
\]
Hence, for each fixed $\boldsymbol\ell\in\mathfrak L_k$,
the set $\mathcal J_{j,\omega,i}(\boldsymbol\ell)$ is a consecutive
segment of $(r_{j,\omega}(v))_{1\le v\le N_{j,\omega}}$. We may therefore take
$\mathcal J=\mathcal J_{j,\omega,i}(\boldsymbol\ell)$ in
\eqref{eq:stage-information}.
The index ranges $Q_{j-1}<r\le Q_j$ are disjoint for different
stages $j$, so summing over $j$ counts each block at most once.
For the fixed length pattern $\boldsymbol\ell$, let $j_k(\boldsymbol\ell)$ be the largest
$j$ for which a stage-$j$ block meets one of the first $k$ positions of the concatenated
sequence.
For $1\le j\le j_k(\boldsymbol\ell)$ and $\omega\in\{\eta,\xi\}$, write
\[
  n_{j,\omega,i}=\#\mathcal J_{j,\omega,i}(\boldsymbol\ell).
\]
If $n_{j,\omega,i}>0$, then a complete stage-$j$ block lies in $I_i(k)$, and hence
$L_j\le D_i\le k$. Therefore
\[ 
  \sum_{j=1}^{j_k(\boldsymbol\ell)}
  \sum_{\omega\in\{\eta,\xi\}}n_{j,\omega,i}\le r_k,
  \qquad
  \log L_j\le\log k,
  \qquad
  \sum_{j=1}^{j_k(\boldsymbol\ell)}L_j
  \lesssim L_*(k),
\]
where the last inequality follows from
$\sum_{j=1}^{j_k(\boldsymbol\ell)}L_j\lesssim L_{j_k(\boldsymbol\ell)}$ and
$L_{j_k(\boldsymbol\ell)}\le L_*(k)$. Moreover,
\[
  \sum_{j=1}^{j_k(\boldsymbol\ell)}n_{j,\eta,i}L_j
  =\ell_{\eta,i,k}(\boldsymbol\ell),
  \qquad
  \sum_{j=1}^{j_k(\boldsymbol\ell)}n_{j,\xi,i}L_j
  =\ell_{\xi,i,k}(\boldsymbol\ell).
\]
For fixed $i$ and $\boldsymbol\ell$, the sets
$\mathcal J_{j,\omega,i}(\boldsymbol\ell)$, over all $j$ and $\omega$, form a partition of
$\mathcal B_{i,k}(\boldsymbol\ell)$. Summing \eqref{eq:stage-information} over $j$ and
$\omega$ therefore gives
\begin{align}
  \sum_{r\in\mathcal B_{i,k}(\boldsymbol\ell)}\log B_{r,i}
  &=\sum_{j=1}^{j_k(\boldsymbol\ell)}
    \sum_{\omega\in\{\eta,\xi\}}
    \sum_{r\in\mathcal J_{j,\omega,i}(\boldsymbol\ell)}
    \log B_{r,i}\notag\\
  &\ge\sum_{j=1}^{j_k(\boldsymbol\ell)}
    \bigl(n_{j,\eta,i}L_jh(\tau_i\eta)
          +n_{j,\xi,i}L_jh(\tau_i\xi)\bigr)
  -O\left(
       \sum_{j=1}^{j_k(\boldsymbol\ell)}
       \sum_{\omega\in\{\eta,\xi\}}n_{j,\omega,i}\log L_j
       +\sum_{j=1}^{j_k(\boldsymbol\ell)}L_j
     \right)\notag\\
  &\ge h(\tau_i\eta)\ell_{\eta,i,k}(\boldsymbol\ell)
       +h(\tau_i\xi)\ell_{\xi,i,k}(\boldsymbol\ell)
  -O\bigl(r_k\log k+L_*(k)\bigr).  \label{eq:summed-stage-information}
\end{align}

The positions in $I_i(k)$ not counted by
$\ell_{\eta,i,k}+\ell_{\xi,i,k}$ lie in joining words or in the parts of at most two
source blocks cut by the endpoints. Thus
\begin{equation}\label{eq:window-source-loss}
  0\le D_i-\ell_{\eta,i,k}(\boldsymbol\ell)
             -\ell_{\xi,i,k}(\boldsymbol\ell)
  \le p\,r_k+2L_*(k).
\end{equation}

For a fixed length pattern $\boldsymbol\ell$, the values $M_{\mathbf I}(t)$, $1\le t\le k$,
are the same for all $\mathbf I\in\Omega_{\boldsymbol\ell}$; denote them by
$M_{\boldsymbol\ell}(t)$, and set $M_{\boldsymbol\ell}(0)=0$.

The difference
$M_{\boldsymbol\ell}(k_i)-M_{\boldsymbol\ell}(k_{i-1})$ counts the $\xi$-block digits in
$I_i(k)=(k_{i-1},k_i]\cap\NN$, including the parts of at most two blocks cut by the endpoints. Hence
\begin{equation}\label{eq:window-xi-loss}
  0\le M_{\boldsymbol\ell}(k_i)-M_{\boldsymbol\ell}(k_{i-1})
       -\ell_{\xi,i,k}(\boldsymbol\ell)
  \le2L_*(k).
\end{equation}
For $1\le i\le s$, write
\[
  e_{0,i}=D_i-\ell_{\eta,i,k}(\boldsymbol\ell)
                 -\ell_{\xi,i,k}(\boldsymbol\ell),
  \qquad
  e_{1,i}=M_{\boldsymbol\ell}(k_i)-M_{\boldsymbol\ell}(k_{i-1})
                 -\ell_{\xi,i,k}(\boldsymbol\ell).
\]
Equations \eqref{eq:window-source-loss} and \eqref{eq:window-xi-loss} give
\[
  0\le e_{0,i}\le p\,r_k+2L_*(k),
  \qquad
  0\le e_{1,i}\le2L_*(k).
\]
Since $h(\tau_i\xi)=h(\tau_i\eta)+\Delta_i$, the two main terms on the right-hand side of
\eqref{eq:summed-stage-information} are exactly
\begin{align*}
 &h(\tau_i\eta)\ell_{\eta,i,k}(\boldsymbol\ell)
   +h(\tau_i\xi)\ell_{\xi,i,k}(\boldsymbol\ell)\\
 &=h(\tau_i\eta)
   \bigl(\ell_{\eta,i,k}(\boldsymbol\ell)
         +\ell_{\xi,i,k}(\boldsymbol\ell)\bigr)
   +\Delta_i\ell_{\xi,i,k}(\boldsymbol\ell)\\
 &=h(\tau_i\eta)(D_i-e_{0,i})
   +\Delta_i\bigl(
      M_{\boldsymbol\ell}(k_i)-M_{\boldsymbol\ell}(k_{i-1})-e_{1,i}
    \bigr)\\
 &=D_i h(\tau_i\eta)
   +\Delta_i\bigl(M_{\boldsymbol\ell}(k_i)-M_{\boldsymbol\ell}(k_{i-1})\bigr)
   -h(\tau_i\eta)e_{0,i}-\Delta_i e_{1,i}.
\end{align*}
Since the entropies and the numbers $\Delta_i$ are fixed,
\eqref{eq:summed-stage-information} and the preceding bounds give
\begin{align}
  \sum_{r\in\mathcal B_{i,k}(\boldsymbol\ell)}\log B_{r,i}
  &\ge D_i h(\tau_i\eta)\notag\\
  &\quad+\Delta_i
    \bigl(M_{\boldsymbol\ell}(k_i)-M_{\boldsymbol\ell}(k_{i-1})\bigr)
    -O\bigl(r_k\log k+L_*(k)+p\,r_k\bigr). \label{eq:window-information}
\end{align}

Applying \eqref{eq:M-tracking} at both
endpoints and using $0\le k_{i-1}\le k_i\le k$ gives
\[
  \left|
  \bigl(M_{\boldsymbol\ell}(k_i)-M_{\boldsymbol\ell}(k_{i-1})\bigr)
  -\bigl(E_\beta(k_i)-E_\beta(k_{i-1})\bigr)
  \right|
  \lesssim L_*(k)+p\,r_k.
\]

Substitution in \eqref{eq:window-information} yields, uniformly in
$\boldsymbol\ell$,
\begin{align}
  \sum_{r\in\mathcal B_{i,k}(\boldsymbol\ell)}\log B_{r,i}
  &\ge D_i h(\tau_i\eta)
  +\Delta_i\bigl(E_\beta(k_i)-E_\beta(k_{i-1})\bigr)\notag\\
  &\quad-O\bigl(r_k\log k+L_*(k)+p\,r_k\bigr).
  \label{eq:window-E-information}
\end{align}

If $\lambda(Q_k(z))=0$, the conclusion is immediate, with
$-\log0=+\infty$. Otherwise, \eqref{eq:geometric-product} gives
\[
  -\log\lambda(Q_k(z))
  \ge
  \min_{\boldsymbol\ell\in\mathfrak L_k}
  \sum_{i=1}^s\sum_{r\in\mathcal B_{i,k}(\boldsymbol\ell)}
  \log B_{r,i}
  -r_k\log(p+1)-d\log2.
\]
For every $\boldsymbol\ell\in\mathfrak L_k$, summing
\eqref{eq:window-E-information} over $i$ yields
\begin{align*}
 \sum_{i=1}^s\sum_{r\in\mathcal B_{i,k}(\boldsymbol\ell)}\log B_{r,i}
 &\ge \sum_{i=1}^sD_i h(\tau_i\eta)\\
 &\quad+\sum_{i=1}^s\Delta_i\bigl(E_\beta(k_i)-E_\beta(k_{i-1})\bigr)
   -O\bigl(r_k\log k+L_*(k)+p\,r_k\bigr).
\end{align*}
The right-hand side is independent of $\boldsymbol\ell$.
Taking the minimum over $\boldsymbol\ell$ and adding the loss
$r_k\log(p+1)+d\log2$ gives a total error of order $O(\mathcal E_k)$.
Now $\mathcal E_k=o(k^\beta)$ by \eqref{eq:total-error-small}, which proves
\eqref{eq:master-information}.
\end{proof}

\subsection{Completion of the proof}

\begin{proof}[Proof of Proposition~\ref{thm:measure-construction}]
Let $\lambda$ be the probability measure constructed in
Lemma~\ref{lem:scheduled-measure}. Set
\begin{equation*}
  A=\sum_{i=1}^s
  \Delta_i(\theta_i^\beta-\theta_{i-1}^\beta)>0.
\end{equation*}
The function $x\mapsto(x+1)^{\beta-1}$ is decreasing. Hence, for every integer $N\ge1$,
\begin{align*}
  \int_1^{N+1}(x+1)^{\beta-1}\,dx
  \le E_\beta(N)
  \le\int_0^N(x+1)^{\beta-1}\,dx.
\end{align*}
Evaluating the two integrals gives
\begin{equation}\label{eq:Ebeta-asymptotic}
  E_\beta(N)=\frac{N^\beta}{\beta}+O(1).
\end{equation}
By \eqref{eq:Ebeta-asymptotic}, the identities
$\lfloor\theta_i k\rfloor=\theta_i k+O(1)$ and $E_\beta(0)=0$ give
\begin{equation}\label{eq:gain-main}
 \sum_{i=1}^s\Delta_i
  \bigl(E_\beta(\lfloor\theta_i k\rfloor)
   -E_\beta(\lfloor\theta_{i-1}k\rfloor)\bigr)
 =\frac{k^\beta}{\beta}
    \sum_{i=1}^s\Delta_i(\theta_i^\beta-\theta_{i-1}^\beta)+O(1)
 =\frac{A}{\beta}k^\beta+O(1).
\end{equation}
Also, by \eqref{eq:theta-a} and condition \textup{(ii)},
\begin{equation}
 \sum_{i=1}^s
  (\lfloor\theta_i k\rfloor-\lfloor\theta_{i-1}k\rfloor)h(\tau_i\eta)
 =k\sum_{i=1}^s(\theta_i-\theta_{i-1})h(\tau_i\eta)+O(1)
 =kP^{\mathbf a}(\sigma)+O(1).           \label{eq:baseline-main}
\end{equation}
Equations \eqref{eq:master-information}, \eqref{eq:gain-main}, and
\eqref{eq:baseline-main} give
\[
  -\log\lambda(Q_k(z))
  \ge kP^{\mathbf a}(\sigma)+\frac{A}{\beta}k^\beta-o(k^\beta).
\]
Here the $o(k^\beta)$ term is uniform for $z\in K$. Hence there is $k_0$ such that,
for every $z\in K$ and $k\ge k_0$, the absolute value of this error is at most
$Ak^\beta/(2\beta)$. Then
\[
  \log\lambda(Q_k(z))+kP^{\mathbf a}(\sigma)
  \le-\frac{A}{2\beta}k^\beta.
\]
Thus \eqref{eq:cube-mass-decay} holds with $c_0=A/(2\beta)$.
\end{proof}

\section{Proof of Proposition~\ref{thm:pressure-positive}}\label{sec:pressure-inequality}

Recall that $\eta=\eta_{\mathbf a}$ is the unique $\mathbf a$-weighted equilibrium state
from Proposition~\ref{prop:known-results}\textup{(ii)}. Thus
$P^{\mathbf a}(\sigma)=\sum_{i=1}^s a_i h(\tau_i\eta)$. We prove that
$\dimH K<\dimB K$ implies $c(\beta)>0$ for some $\beta\in(1/2,1)$. Suppose instead that
$c$ vanishes on an open interval. We show that $\eta$ is then the equilibrium state for an
open family of weights. The resulting uniform fibre estimates imply
$\dimH K=\dimB K$, a contradiction.

\subsection{Common equilibrium states}

We first show that if $c$ vanishes on an open interval, then $\eta$ is the
equilibrium state for an open family of weights.

For a positive probability weight vector $\mathbf b=(b_1,\ldots,b_s)$, write
\[
  A_i(\mathbf b)=\sum_{\ell=1}^i b_\ell,
  \qquad 0<A_1(\mathbf b)<\cdots<A_s(\mathbf b)=1.
\]
For $\mathbf b=\mathbf a(\beta)$, where $\mathbf a(\beta)$ is defined in
\eqref{eq:power-weight},
\begin{equation}\label{eq:power-cumulative}
  A_i(\mathbf a(\beta))=\theta_i^\beta.
\end{equation}

\begin{lemma}\label{lem:affine-span}
Let $I\subset(0,1)$ be a nonempty open interval. There are distinct
$\beta_1,\ldots,\beta_s\in I$ such that
$\mathbf a(\beta_1),\ldots,\mathbf a(\beta_s)$ are affinely independent. More explicitly,
if real numbers $t_1,\ldots,t_s$ satisfy
\[
  \sum_{j=1}^s t_j=0,
  \qquad
  \sum_{j=1}^s t_j\mathbf a(\beta_j)=0,
\]
then $t_1=\cdots=t_s=0$.
\end{lemma}

\begin{proof}
Write $m=s-1$.
We regard vectors in $\mathbb R^m$ as column vectors. The map from $\mathbf b$ to
$
  \bigl(A_1(\mathbf b),\ldots,A_m(\mathbf b)\bigr)^{\mathsf T}
$
is affine and one-to-one. Consider the column vector
\[
  F(\beta)=\bigl(\theta_1^\beta,\ldots,\theta_m^\beta\bigr)^{\mathsf T}.
\]
Fix $\beta_0\in I$, and write $x_i=\log\theta_i$ for $1\le i\le m$.
For $1\le r\le m$, the $r$-th derivative of $F$ at $\beta_0$ is
\[
  F^{(r)}(\beta_0)
  =\bigl(\theta_1^{\beta_0}x_1^r,\ldots,
         \theta_m^{\beta_0}x_m^r\bigr)^{\mathsf T}.
\]
Consequently, factoring $\theta_i^{\beta_0}x_i$ from row $i$ gives
\begin{align}
 \det\bigl(F'(\beta_0),\ldots,F^{(m)}(\beta_0)\bigr)
 &=\prod_{i=1}^m\theta_i^{\beta_0}x_i \cdot 
   \det(x_i^{r-1})_{1\le i,r\le m}\notag\\
 &=\prod_{i=1}^m\theta_i^{\beta_0}\log\theta_i
   \prod_{1\le i<j\le m}(\log\theta_j-\log\theta_i).       \label{eq:derivative-vandermonde}
\end{align}
Since $0<\theta_1<\cdots<\theta_m<1$, \eqref{eq:derivative-vandermonde} is nonzero.

Choose distinct real numbers $u_1,\ldots,u_s$. For $2\le j\le s$, Taylor's formula at
$\beta_0$ gives
\begin{equation}
 F(\beta_0+\varepsilon u_j)-F(\beta_0+\varepsilon u_1)
 =\sum_{r=1}^m\frac{\varepsilon^r(u_j^r-u_1^r)}{r!}F^{(r)}(\beta_0)
  +O(\varepsilon^{m+1}).                              \label{eq:F-difference-expansion}
\end{equation}
Let $C$ be the $m\times m$ matrix whose entry in row $r$ and column $j-1$ is
$(u_j^r-u_1^r)/r!$. Subtracting the first column from every other column in the usual
Vandermonde determinant for $u_1,\ldots,u_s$, and then expanding along its first row, gives
\begin{equation}\label{eq:parameter-vandermonde}
 \det C
 =\frac{1}{\prod_{r=1}^m r!}
   \prod_{1\le i<j\le s}(u_j-u_i)\ne0.
\end{equation}
By the multilinearity of the determinant, \eqref{eq:F-difference-expansion} therefore gives
\begin{align}
 &\det\bigl(
 F(\beta_0+\varepsilon u_2)-F(\beta_0+\varepsilon u_1),\ldots,
 F(\beta_0+\varepsilon u_s)-F(\beta_0+\varepsilon u_1)
 \bigr)\notag\\
 &\qquad
   =\varepsilon^{m(m+1)/2}
    \det\bigl(F'(\beta_0),\ldots,F^{(m)}(\beta_0)\bigr)\det C
    +o\bigl(\varepsilon^{m(m+1)/2}\bigr).                 \label{eq:difference-determinant}
\end{align}
The coefficient of the leading term is nonzero by
\eqref{eq:derivative-vandermonde} and \eqref{eq:parameter-vandermonde}. Thus the determinant
in \eqref{eq:difference-determinant} is nonzero for every sufficiently small positive
$\varepsilon$. We may also choose $\varepsilon$ so small that
$\beta_j=\beta_0+\varepsilon u_j$ belongs to $I$ for every $1\le j\le s$. Since this
determinant is nonzero, the vectors $F(\beta_j)-F(\beta_1)$, $2\le j\le s$, are linearly
independent. Hence the points $F(\beta_1),\ldots,F(\beta_s)$ are affinely independent.

Finally,
\eqref{eq:power-cumulative} says that the affine one-to-one map at the beginning of the proof
sends $\mathbf a(\beta_j)$ to $F(\beta_j)$. Hence
$\mathbf a(\beta_1),\ldots,\mathbf a(\beta_s)$ are affinely independent.
\end{proof}

\begin{lemma}\label{lem:common-open}
Suppose $c(\beta)=0$ for every $\beta$ in a nonempty open interval $I$, and let
$\eta=\eta_{\mathbf a}$. Then there is a
relatively open subset $U$ of the positive weight simplex such that $\eta$ is the unique
$\mathbf b$-weighted equilibrium state for every $\mathbf b\in U$.
\end{lemma}

\begin{proof}
Choose $\beta_1,\ldots,\beta_s$ as in Lemma~\ref{lem:affine-span}. The equality
$c(\beta_j)=0$ means that $\eta$ attains the supremum defining
$P^{\mathbf a(\beta_j)}(\sigma)$. If
\[
  \mathbf b=\sum_{j=1}^s t_j\mathbf a(\beta_j),
  \qquad t_j>0,\qquad \sum_jt_j=1,
\]
then for every $\mu\in\cM_\sigma(X)$,
\begin{align*}
  \sum_i b_i h(\tau_i\mu)
  &=\sum_jt_j\sum_i a_i(\beta_j)h(\tau_i\mu)\\
  &\le\sum_jt_j\sum_i a_i(\beta_j)h(\tau_i\eta)
   =\sum_i b_i h(\tau_i\eta).
\end{align*}
Thus $\eta$ is an equilibrium state for every weight vector in the relative interior
of the convex hull of these vectors.
Uniqueness follows from Proposition~\ref{prop:known-results}.
\end{proof}

\subsection{Recursive-function estimate}

Combining the common equilibrium state given by Lemma~\ref{lem:common-open} with the
Gibbs estimates of D.-J. Feng \cite{Feng2011}, we will obtain a recursive-function estimate.

For a positive probability weight vector $\mathbf b=(b_1,\ldots,b_s)$, set
\begin{equation}\label{eq:ratio-coordinates}
  \alpha_1(\mathbf b)=0,
  \qquad
  \alpha_i(\mathbf b)=\frac{A_{i-1}(\mathbf b)}{A_i(\mathbf b)}
  \quad 2\le i\le s.
\end{equation}
When $\mathbf b$ is fixed, we write simply $\alpha_i$ for
$\alpha_i(\mathbf b)$.
For the vector $\mathbf a$ in
\eqref{eq:theta-a}, this is exactly
\[
  \alpha_i(\mathbf a)=\frac{\theta_{i-1}}{\theta_i}
  =\frac{\log n_i}{\log n_{i-1}}.
\]
The map $\mathbf b\mapsto(\alpha_2,\ldots,\alpha_s)$ is a diffeomorphism from the positive
weight simplex onto $(0,1)^{s-1}$. Let $V$ be the image of the set $U$ from
Lemma~\ref{lem:common-open} under this map. The open set
$V$ contains an open box
\begin{equation}\label{eq:ratio-box}
  B=I_2\times\cdots\times I_s,
  \qquad
  \overline B\subset V\subset(0,1)^{s-1}.
\end{equation}
For $\boldsymbol\alpha\in B$, let $\mathbf b(\boldsymbol\alpha)\in U$ denote the unique
weight vector given by the inverse diffeomorphism.

We now define recursive functions associated with the successive factor maps. Their
definition agrees with the functions used in D.-J. Feng's weighted equilibrium construction, so
the Gibbs estimates in \cite{Feng2011} can be applied. Fix $\mathbf b$, set
$\phi_{\mathbf b}^{(1)}\equiv1$ on $\cL(X_1)$, and define the
remaining functions recursively by

\begin{equation}\label{eq:recursive-phi}
\begin{aligned}
  \phi_{\mathbf b}^{(2)}(J)
  &=\#\{I\in\cL_{|J|}(X_1):\pi_2(I)=J\},
  &&J\in\cL(X_2),\\
  \phi_{\mathbf b}^{(i)}(J)
  &=\sum_{\substack{I\in\cL_{|J|}(X_{i-1})\\ \pi_i(I)=J}}
   \phi_{\mathbf b}^{(i-1)}(I)^{\alpha_{i-1}(\mathbf b)},
  &&J\in\cL(X_i),\quad 3\le i\le s.
\end{aligned}
\end{equation}
For $N\ge1$, write
\begin{equation}\label{eq:Zr}
  Z_{\mathbf b}(N)=
  \sum_{J\in\cL_N(X_s)}
  \phi_{\mathbf b}^{(s)}(J)^{\alpha_s(\mathbf b)}.
\end{equation}
The function $\phi_{\mathbf b}^{(i)}$ depends only on
$\alpha_2(\mathbf b),\ldots,\alpha_{i-1}(\mathbf b)$.

The following proposition is based on \cite{Feng2011}.

\begin{proposition}\label{prop:feng-input}
Let $\eta_{\mathbf b}$ be the $\mathbf b$-weighted equilibrium state. Then:
\begin{enumerate}[label=\textup{(\roman*)}]
  \item for every $N\ge1$ and $J\in\cL_N(X_s)$,
  \begin{equation}\label{eq:coarse-gibbs}
    (\tau_s\eta_{\mathbf b})[J]\asymp_{\mathbf b}
    \frac{\phi_{\mathbf b}^{(s)}(J)^{\alpha_s(\mathbf b)}}
         {Z_{\mathbf b}(N)};
  \end{equation}
  \item if $1\le j<s$, $N\ge1$, and $I_j\in\cL_N(X_j)$, let
  $I_{k}=\pi_k\circ\cdots\circ\pi_{j+1}(I_j)$ for $j<k\le s$, and write
  \begin{equation}\label{eq:raw-weight}
    \widetilde g_{j,\mathbf b,N}(I_j)
    =
    \left(
    \prod_{k=j}^{s-1}
    \frac{\phi_{\mathbf b}^{(k)}(I_k)^{\alpha_k(\mathbf b)}}
         {\phi_{\mathbf b}^{(k+1)}(I_{k+1})}
    \right)
    \frac{\phi_{\mathbf b}^{(s)}(I_s)^{\alpha_s(\mathbf b)}}
         {Z_{\mathbf b}(N)}.
  \end{equation}
  Then
  \begin{equation}\label{eq:one-sided-gibbs}
    \sum_{I_j\in\cL_N(X_j)}\widetilde g_{j,\mathbf b,N}(I_j)=1,
    \qquad
    \widetilde g_{j,\mathbf b,N}(I_j)
    \lesssim_{\mathbf b}(\tau_j\eta_{\mathbf b})[I_j]. 
  \end{equation}
\end{enumerate}
All implicit constants in this proposition are independent of $N$ and of the words involved.
\end{proposition}

\begin{proof}
For a subshift $Y$ satisfying weak specification with bound $p$, recall from
\cite[Section~5]{Feng2011} that $\mathcal D_w(Y,p)$ is the collection of positive word
functions $\psi:\cL(Y)\to(0,\infty)$ satisfying the following two properties: for some
$C\ge1$,
\[
  \psi(IJ)\le C\psi(I)\psi(J)
  \quad\text{whenever }IJ\in\cL(Y),
\]
and, for every $I,J\in\cL(Y)$, there is a word $W\in\cL(Y)$ with $|W|\leq p$ such that
\[
  IWJ\in\cL(Y),
  \qquad
  \psi(IWJ)\ge C^{-1}\psi(I)\psi(J).
\]
The function $\phi_{\mathbf b}^{(1)}\equiv1$ on $\cL(X_1)$ belongs to
$\mathcal D_w(X_1,p)$. By \cite[Lemma~5.7]{Feng2011}, every $\phi_{\mathbf b}^{(i)}$,
and every positive power of it, belongs to $\mathcal D_w(X_i,p)$, since positive powers and
fibre sums under a factor map preserve the required two inequalities.

For part \textup{(i)}, \cite[Lemma~7.2]{Feng2011} identifies $\tau_s\eta_{\mathbf b}$
as the equilibrium state on $X_s$ generated by
$\bigl(\phi_{\mathbf b}^{(s)}\bigr)^{\alpha_s(\mathbf b)}$. Thus
\cite[Theorem~5.5]{Feng2011}, applied to this word function, gives
\eqref{eq:coarse-gibbs}.

For part \textup{(ii)}, \cite[Theorem~7.3(i)]{Feng2011} gives the required inequality
when $j=1$. Now fix $2\le j<s$. The recursion \eqref{eq:recursive-phi} says that, for every $1\le k<s$ and
$I_{k+1}\in\cL_N(X_{k+1})$,
\[
  \sum_{\pi_{k+1}(I_k)=I_{k+1}}
  \frac{\phi_{\mathbf b}^{(k)}(I_k)^{\alpha_k(\mathbf b)}}
       {\phi_{\mathbf b}^{(k+1)}(I_{k+1})}=1.
\]
Consequently, successively summing the preceding expression for
$\widetilde g_{1,\mathbf b,N}$ over the fibres of
$\pi_2,\ldots,\pi_j$ gives
\begin{equation}\label{eq:raw-weight-pushforward}
  \widetilde g_{j,\mathbf b,N}(I_j)
  =\sum_{\substack{I_1\in\cL_N(X_1)\\ (\pi_j\circ\cdots\circ\pi_2)(I_1)=I_j}}
    \widetilde g_{1,\mathbf b,N}(I_1).
\end{equation}
Using the definition \eqref{eq:Zr}, the same fibre-sum calculation, carried all
the way to level $s$, gives
\[
  \sum_{I_1\in\cL_N(X_1)}\widetilde g_{1,\mathbf b,N}(I_1)
  =\frac{1}{Z_{\mathbf b}(N)}
    \sum_{I_s\in\cL_N(X_s)}
    \phi_{\mathbf b}^{(s)}(I_s)^{\alpha_s(\mathbf b)}=1.
\]
Summing \eqref{eq:raw-weight-pushforward} over $I_j\in\cL_N(X_j)$ gives
\[
  \sum_{I_j\in\cL_N(X_j)}\widetilde g_{j,\mathbf b,N}(I_j)
  =\sum_{I_1\in\cL_N(X_1)}\widetilde g_{1,\mathbf b,N}(I_1)=1,
\]
which proves the first assertion in \eqref{eq:one-sided-gibbs}. Finally, using the $j=1$
inequality and \eqref{eq:raw-weight-pushforward} gives
\[
  \widetilde g_{j,\mathbf b,N}(I_j)
  \lesssim_{\mathbf b}
  \sum_{(\pi_j\circ\cdots\circ\pi_2)(I_1)=I_j}\eta_{\mathbf b}[I_1]
  = (\tau_j\eta_{\mathbf b})[I_j],
\]
which proves the second assertion in \eqref{eq:one-sided-gibbs}.
\end{proof}

For a subshift $Y$, we call a positive function $\phi$ on $\cL(Y)$ \emph{comparable} if
there is a constant $C\ge1$ such that
\begin{equation}\label{eq:equal-length-uniform}
  C^{-1}\le\frac{\phi(I)}{\phi(J)}\le C
  \quad\text{whenever }I,J\in\cL_N(Y),\ N\ge1,
\end{equation}
where the constant $C$ is independent of $N,I,J$.

\begin{lemma}\label{lem:gibbs-upgrade}
Fix $1\le j<s$ and $\mathbf b$. Suppose
$\phi_{\mathbf b}^{(j+1)},\ldots,\phi_{\mathbf b}^{(s)}$ satisfy
\eqref{eq:equal-length-uniform}. Then, for every $N\ge1$ and
$I\in\cL_N(X_j)$,
\begin{equation}\label{eq:upgraded-gibbs}
  (\tau_j\eta_{\mathbf b})[I]\asymp_{\mathbf b}
  \frac{\phi_{\mathbf b}^{(j)}(I)^{\alpha_j(\mathbf b)}}
       {\displaystyle\sum_{L\in\cL_N(X_j)}
         \phi_{\mathbf b}^{(j)}(L)^{\alpha_j(\mathbf b)}}.
\end{equation}
\end{lemma}

\begin{proof}
Fix $N\ge1$. For each $j+1\le k\le s$, choose a word
$J_{k,N}\in\cL_N(X_k)$. Condition \eqref{eq:equal-length-uniform} gives
\[
  \phi_{\mathbf b}^{(k)}(I_k)
  \asymp_{\mathbf b}
  \phi_{\mathbf b}^{(k)}(J_{k,N})
  \qquad I_k\in\cL_N(X_k).
\]
Substitution in \eqref{eq:raw-weight} gives
\[
  \widetilde g_{j,\mathbf b,N}(I)
  \asymp_{\mathbf b}
  T_N\phi_{\mathbf b}^{(j)}(I)^{\alpha_j(\mathbf b)},
  \qquad
  T_N=\frac{1}{Z_{\mathbf b}(N)}
  \prod_{k=j+1}^s
  \phi_{\mathbf b}^{(k)}(J_{k,N})^{\alpha_k(\mathbf b)-1}.
\]
The quantity $T_N$ may depend on $N$, but it does not depend on $I$. The normalisation in
\eqref{eq:one-sided-gibbs} now gives
\[
  1=\sum_{I\in\cL_N(X_j)}\widetilde g_{j,\mathbf b,N}(I)
  \asymp_{\mathbf b}
  T_N\sum_{I\in\cL_N(X_j)}
  \phi_{\mathbf b}^{(j)}(I)^{\alpha_j(\mathbf b)}.
\]
Solving this comparison for $T_N$ and substituting it into the preceding estimate gives
\begin{equation}\label{eq:raw-normalised}
  \widetilde g_{j,\mathbf b,N}(I)
  \asymp_{\mathbf b}
  \frac{\phi_{\mathbf b}^{(j)}(I)^{\alpha_j(\mathbf b)}}
       {\displaystyle\sum_{L\in\cL_N(X_j)}
        \phi_{\mathbf b}^{(j)}(L)^{\alpha_j(\mathbf b)}}.
\end{equation}
If $j=1$, then $\bigl(\phi_{\mathbf b}^{(1)}\bigr)^{\alpha_1(\mathbf b)}=1$, which belongs
to $\mathcal D_w(X_1,p)$, as noted in the proof of Proposition~\ref{prop:feng-input}. If $2\le j<s$, the verification in the proof of
Proposition~\ref{prop:feng-input} shows that
$\bigl(\phi_{\mathbf b}^{(j)}\bigr)^{\alpha_j(\mathbf b)}$ belongs to
$\mathcal D_w(X_j,p)$. Thus \cite[Theorem~5.5]{Feng2011} gives a unique equilibrium state
$\nu$ on $X_j$. Moreover, $\nu$ is ergodic and satisfies
\begin{equation}\label{eq:nu-gibbs}
  \nu[I]\asymp_{\mathbf b}
  \frac{\phi_{\mathbf b}^{(j)}(I)^{\alpha_j(\mathbf b)}}
       {\displaystyle\sum_{L\in\cL_N(X_j)}
        \phi_{\mathbf b}^{(j)}(L)^{\alpha_j(\mathbf b)}}.
\end{equation}
Equations \eqref{eq:one-sided-gibbs}, \eqref{eq:raw-normalised}, and
\eqref{eq:nu-gibbs} imply that
\[
  \nu[I]\lesssim_{\mathbf b}(\tau_j\eta_{\mathbf b})[I]
\]
for every cylinder $[I]\subset X_j$. Every open subset of $X_j$ can be written as a
countable disjoint union of cylinders, so the same inequality holds for open sets. By
regularity of Borel probability measures, it follows that
$\nu\ll\tau_j\eta_{\mathbf b}$.

By \cite[Theorem~7.3]{Feng2011}, the weighted equilibrium state $\eta_{\mathbf b}$ is
ergodic, and hence its factor $\tau_j\eta_{\mathbf b}$ is ergodic. Since $\nu$ is also
ergodic, distinct ergodic invariant probability measures are
mutually singular, whereas $\nu\ll\tau_j\eta_{\mathbf b}$. Therefore
$\nu=\tau_j\eta_{\mathbf b}$, and \eqref{eq:nu-gibbs} proves
\eqref{eq:upgraded-gibbs}.
\end{proof}

We now combine the preceding estimates with the fact that the equilibrium state is the same
throughout the box $B$. By changing one coordinate $\alpha_i$ while keeping the others
fixed, two Gibbs estimates for the same measure can be compared.

\begin{lemma}\label{lem:downward}
Suppose that, for every $\boldsymbol\alpha=(\alpha_2,\ldots,\alpha_s)$ in the box
$B=I_2\times\cdots\times I_s$ defined in \eqref{eq:ratio-box}, the same measure $\eta$ is
the $\mathbf b(\boldsymbol\alpha)$-weighted equilibrium state. Then,
for every fixed $\boldsymbol\alpha\in B$ and $1\le i\le s$, the function
$\phi_{\mathbf b(\boldsymbol\alpha)}^{(i)}$ satisfies
\eqref{eq:equal-length-uniform}.
\end{lemma}

\begin{proof}
The case $i=1$ is immediate since $\phi_{\mathbf b(\boldsymbol\alpha)}^{(1)}\equiv1$.
We argue from $i=s$ down to $i=2$. First fix
$\alpha_2,\ldots,\alpha_{s-1}$ and choose distinct $u,v\in I_s$. For $t\in\{u,v\}$,
let $\mathbf b_t$ be the weight vector corresponding to
$(\alpha_2,\ldots,\alpha_{s-1},t)$. The function
$\phi_{\mathbf b_t}^{(s)}$ does not depend on $t$; denote this common function by $\Phi_s$.
The measure $\tau_s\eta$ is also unchanged. Write
\[
  S_{s,t}(N)=\sum_{L\in\cL_N(X_s)}\Phi_s(L)^t,
  \qquad t\in\{u,v\}
\]
and apply \eqref{eq:coarse-gibbs} to $\mathbf b_u$ and $\mathbf b_v$. This gives
\[
  \frac{\Phi_s(J)^u}{S_{s,u}(N)}
  \asymp_{\mathbf b_u}(\tau_s\eta)[J]
  \asymp_{\mathbf b_v}
  \frac{\Phi_s(J)^v}{S_{s,v}(N)}.
\]
Therefore
\[
  \Phi_s(J)^{u-v}
  \asymp_{\mathbf b_u,\mathbf b_v}
  \frac{S_{s,u}(N)}{S_{s,v}(N)}.
\]
The comparison constants are independent of $N$ and $J$. Hence, for
$J,L\in\cL_N(X_s)$,
\[
  \left(\frac{\Phi_s(J)}{\Phi_s(L)}\right)^{u-v}
  \asymp_{\mathbf b_u,\mathbf b_v}1,
\]
with an implicit constant independent of $N,J,L$. Since $u\ne v$, this proves
\eqref{eq:equal-length-uniform} for
$\phi_{\mathbf b(\boldsymbol\alpha)}^{(s)}$ at every $\boldsymbol\alpha\in B$.

Now suppose the assertion has been proved for
$\phi_{\mathbf b}^{(j+1)},\ldots,\phi_{\mathbf b}^{(s)}$ at every point of $B$.
Fix all coordinates of $\boldsymbol\alpha$ except $\alpha_j$, and choose distinct
$u,v\in I_j$. For $t\in\{u,v\}$, let
$\mathbf b_t$ be the weight vector obtained by setting $\alpha_j=t$. The function
$\phi_{\mathbf b_t}^{(j)}$ depends only on $\alpha_2,\ldots,\alpha_{j-1}$ and is therefore
independent of $t$; denote it by $\Phi_j$. Lemma~\ref{lem:gibbs-upgrade}, applied to
$\mathbf b_u$ and $\mathbf b_v$, gives
\[
  (\tau_j\eta)[I]
  \asymp_{\mathbf b_u}
  \frac{\Phi_j(I)^u}{S_{j,u}(N)},
  \qquad
  (\tau_j\eta)[I]
  \asymp_{\mathbf b_v}
  \frac{\Phi_j(I)^v}{S_{j,v}(N)},
\]
where $S_{j,t}(N)=\sum_{L\in\cL_N(X_j)}\Phi_j(L)^t$. Since
$\tau_j\eta$ is unchanged,
\[
  \Phi_j(I)^{u-v}
  \asymp_{\mathbf b_u,\mathbf b_v}
  \frac{S_{j,u}(N)}{S_{j,v}(N)}.
\]
Again the comparison constants are independent of $N$ and $I$. Thus, for
$I,L\in\cL_N(X_j)$,
\[
  \left(\frac{\Phi_j(I)}{\Phi_j(L)}\right)^{u-v}
  \asymp_{\mathbf b_u,\mathbf b_v}1,
\]
with an implicit constant independent of $N,I,L$. Since $u\ne v$, this proves
\eqref{eq:equal-length-uniform} for
$\phi_{\mathbf b(\boldsymbol\alpha)}^{(j)}$ at every $\boldsymbol\alpha\in B$.
\end{proof}

\subsection{Fibre estimates}

We now pass from the recursive functions to the fibres of the factor maps. We first treat
each adjacent map $\pi_i$ and then each composite map $\tau_i$. The resulting estimates
imply dimension equality and complete the contradiction argument.

For $N\ge1$, $1\le i\le s$, and $J\in\cL_N(X_i)$, write
\[
  f_1(J)=1,
  \qquad
  f_i(J)=\#\{I\in\cL_N(X_{i-1}):\pi_i(I)=J\}\quad(2\le i\le s).
\]

\begin{lemma}\label{lem:adjacent-fibres}
For $2\le i\le s$, if $\phi_{\mathbf b}^{(i-1)}$ and $\phi_{\mathbf b}^{(i)}$ satisfy
\eqref{eq:equal-length-uniform}, then so does $f_i$. The function $f_1$ also satisfies
\eqref{eq:equal-length-uniform}.
\end{lemma}

\begin{proof}
The constant function $f_1\equiv1$ on $\cL(X_1)$ trivially satisfies
\eqref{eq:equal-length-uniform}. Let $2\le i\le s$, and write
\[
  m_N=\min_{I\in\cL_N(X_{i-1})}\phi_{\mathbf b}^{(i-1)}(I),
  \qquad
  M_N=\max_{I\in\cL_N(X_{i-1})}\phi_{\mathbf b}^{(i-1)}(I).
\]
Condition \eqref{eq:equal-length-uniform} gives
$M_N/m_N\le C_{i-1}$, where $C_{i-1}\ge1$ is independent of $N$. By \eqref{eq:recursive-phi},
\[
  f_i(J)m_N^{\alpha_{i-1}}
  \le\phi_{\mathbf b}^{(i)}(J)
  \le f_i(J)M_N^{\alpha_{i-1}}.
\]
Comparing this estimate for two words $J,J'\in\cL_N(X_i)$ proves that
$f_i(J)/f_i(J')$ is bounded above and below independently of $N,J,J'$.
\end{proof}

For $N\ge1$, $1\le i\le s$, and $J\in\cL_N(X_i)$, set
\[
  F_i(J)=\#\{I\in\cL_N(X):\tau_i(I)=J\}.
\]

\begin{lemma}\label{lem:composite-fibres}
If every $f_i$, $1\le i\le s$, satisfies \eqref{eq:equal-length-uniform}, then every
$F_i$ satisfies the same condition.
\end{lemma}

\begin{proof}
The assertion is clear for $F_1\equiv1$. If it holds for $F_i$, then
\[
  F_{i+1}(J)
  =\sum_{\substack{L\in\cL_N(X_i)\\\pi_{i+1}(L)=J}}F_i(L).
\]
This lies between
$f_{i+1}(J)\min_LF_i(L)$ and $f_{i+1}(J)\max_LF_i(L)$. Since both $F_i$ and
$f_{i+1}$ have uniformly comparable values on words of the same length, the same is true of
$F_{i+1}$.
\end{proof}

\begin{lemma}\label{lem:dimension-equality}
If all the functions $F_i$ satisfy \eqref{eq:equal-length-uniform}, then
$\dimH K=\dimB K$.
\end{lemma}

\begin{proof}
The subshifts $X$ and $X_i$ satisfy weak specification. Hence
\cite[Proposition~2.2]{Feng2026} gives
\[
  \#\cL_N(X)\asymp e^{Nh_{\mathrm{top}}(X)},
  \qquad
  \#\cL_N(X_i)\asymp e^{Nh_{\mathrm{top}}(X_i)}.
\]
For $J\in\cL_N(X_i)$, condition \eqref{eq:equal-length-uniform} and the identity
$\sum_{L\in\cL_N(X_i)}F_i(L)=\#\cL_N(X)$ give
\[
  F_i(J)
  \asymp\frac{\#\cL_N(X)}{\#\cL_N(X_i)}
  \asymp e^{N(h_{\mathrm{top}}(X)-h_{\mathrm{top}}(X_i))}.
\]
Thus condition \textup{(c)} of \cite[Theorem~3.1]{Feng2026} holds, and that theorem gives
$\dimH K=\dimB K$.
\end{proof}

\begin{proof}[Proof of Proposition~\ref{thm:pressure-positive}]
Recall from \eqref{eq:pressure-gap} that $c(\beta)\ge0$ for every $\beta\in(0,1)$.
Suppose, to the contrary, that there is no $\beta\in(1/2,1)$ for which $c(\beta)>0$.
Then
\begin{equation}\label{eq:pressure-gap-zero}
  c(\beta)=0
  \qquad\text{for every }\beta\in(1/2,1).
\end{equation}
Apply Lemma~\ref{lem:common-open} with $I=(1/2,1)$. There is a relatively open set $U$ of
positive weight vectors such that the fixed measure $\eta=\eta_{\mathbf a}$ is the unique
$\mathbf b$-weighted equilibrium state for every $\mathbf b\in U$.
Under the ratio coordinates in \eqref{eq:ratio-coordinates}, the image of $U$ contains
the box $B=I_2\times\cdots\times I_s$ from \eqref{eq:ratio-box}. Fix
$\boldsymbol\alpha\in B$ and set $\mathbf b=\mathbf b(\boldsymbol\alpha)$.
Lemma~\ref{lem:downward} shows that
$\phi_{\mathbf b}^{(i)}$ satisfies \eqref{eq:equal-length-uniform} for every
$1\le i\le s$. Lemmas~\ref{lem:adjacent-fibres} and~\ref{lem:composite-fibres} then show
that every $F_i$ satisfies the same condition. Lemma~\ref{lem:dimension-equality}
gives $\dimH K=\dimB K$, contrary to the hypothesis $\dimH K<\dimB K$. Hence
\eqref{eq:pressure-gap-zero} is impossible, and $c(\beta)>0$ for some
$\beta\in(1/2,1)$.
\end{proof}

\section{Proof of Theorem~\ref{thm:reducible}}\label{sec:reducible}

We prove Theorem~\ref{thm:reducible} in three steps. First, for a sofic subshift
$X$, we establish the existence and formula for the box dimension of its corresponding
self-affine set $K=R(X)$. We then treat separately the cases $\dimH K=\dimB K$ and
$\dimH K<\dimB K$. In the first case, the critical Hausdorff measure is positive and
$\sigma$-finite but need not be finite. In the second case, examples show that it may be
positive and finite or infinite and $\sigma$-finite.

Fix a finite vertex-labelled directed graph $G=(V,E,\ell)$ presenting $X$, where $V$ is the
vertex set, $E\subset V\times V$ is the edge set, and $\ell:V\to\cS\subset\cA$ is the label
map. Write
\[
  \Sigma_G=\{v=(v_k)_{k\ge1}\in V^{\NN}:(v_k,v_{k+1})\in E
       \text{ for every }k\ge1\}.
\]
Applying $\ell$ coordinatewise, the graph presentation gives
\[
  X=\ell(\Sigma_G)
   =\{(\ell(v_k))_{k\ge1}:v\in\Sigma_G\}.
\]
For a finite vertex path $v_1\cdots v_N$, we also write
$\ell(v_1\cdots v_N)=\ell(v_1)\cdots\ell(v_N)$. We use the following convention for
directed paths and strongly connected components; compare \cite{MauldinWilliams1988}.

A \textit{directed path} from $u$ to $v$ is a finite sequence of vertices
$u=v_0,v_1,\ldots,v_r=v$, where $r\ge1$, such that
$(v_{\ell-1},v_\ell)\in E$ for every $1\le\ell\le r$.
A \textit{directed cycle} is a directed path whose first and last vertices are the same. We
say that $v$ is \textit{reachable} from $u$ if there is a directed path from $u$ to $v$.

A nonempty set $C\subset V$ is \emph{strongly connected} if, for every $u,v\in C$, there is
a directed path from $u$ to $v$. A \emph{strongly connected component} is a maximal strongly
connected subset in $V$. A vertex belongs to one of these components exactly when it lies on a directed
cycle. Distinct components are disjoint.

An ordered tuple $(D_1,\ldots,D_r)$ of strongly connected components is called a
\emph{component chain} if, for every $1\le j<r$, either $D_j=D_{j+1}$ or a vertex of
$D_{j+1}$ is reachable from a vertex of $D_j$.

Every sequence in $X$ is the label sequence of at least one infinite path in $G$. The
vertices that occur infinitely often along such a path lie in one strongly connected
component, and the path eventually remains in that component. Let $C_1,\ldots,C_M$ be all
the strongly connected components of $G$. For a component $C$, define
\[
  \Sigma_C=\{v\in\Sigma_G:v_k\in C\text{ for every }k\ge1\},
  \qquad X_C=\ell(\Sigma_C),
  \qquad K_C=R(X_C).
\]
The path shift $\Sigma_C$ is irreducible, so its factor $X_C$ satisfies weak specification.

For a strongly connected component $C$ and $1\le i\le s$, set
\begin{equation}\label{eq:component-entropy}
  h_i(C)=h_{\mathrm{top}}\bigl(\tau_i(X_C)\bigr).
\end{equation}
For $1\le j\le M$, let
\[
  \mathcal U_j=\{u\in\cL(X):uy\in X\text{ for some }y\in X_{C_j}\}.
\]
This is a countable set and contains the empty word. If $u=d_1\cdots d_q\in\mathcal U_j$,
define
\[
  f_u(z)=\sum_{\ell=1}^{q}\Lambda^{-\ell}d_\ell+\Lambda^{-q}z
\]
for $z\in\RR^d$. For every $1\le j\le M$, we have $K_{C_j}\subset K$. Every $x\in X$ can be written as $x=uy$, where
$u\in\mathcal U_j$ and $y\in X_{C_j}$ for some $j$, and
\eqref{eq:geometric-coding-map} gives
$R(uy)=f_u(R(y))$. Thus
\begin{equation}\label{eq:scc-decomposition}
  \bigcup_{j=1}^M K_{C_j}
  \subseteq K
  \subseteq
  \bigcup_{j=1}^M\ \bigcup_{u\in\mathcal U_j}f_u(K_{C_j}).
\end{equation}
Every $f_u$ is a bi-Lipschitz affine embedding, so countable stability of Hausdorff dimension
gives
\begin{equation}\label{eq:reducible-Hdim}
  \dimH K=\max_j\dimH K_{C_j}.
\end{equation}

\subsection{Box dimension}

The finite graph presentation will be used to prove that the box dimension exists and to
derive its formula. This finiteness assumption is essential: for general subshifts, even the
box dimension may fail to exist \cite{Jurga2023}.

We now estimate the number of level-$k$ approximate cubes. By
\eqref{eq:number-approximate-cubes}, this number is $\#\calD_k(X)$.
Recall from \eqref{eq:word-vectors} that
\[
\begin{aligned}
  \calD_k(X)=\biggl\{(J_1,\ldots,J_s)\in
  \prod_{i=1}^s\cL_{\lfloor\theta_i k\rfloor-
  \lfloor\theta_{i-1}k\rfloor}(X_i):{}&
  \text{there is }I\in\cL_k(X)\text{ such that}\\
  &\tau_i(I|_{\lfloor\theta_{i-1}k\rfloor}^{\lfloor\theta_i k\rfloor})=J_i
  \text{ for every }i\biggr\}.
\end{aligned}
\]

For a component $C$, $1\le i\le s$, and $N\ge1$, a vertex path
$\widehat J=v_1\cdots v_N\in\cL_N(\Sigma_C)$ is called a \emph{lift} of
$J\in\cL_N(\tau_i(X_C))$ if $\tau_i(\ell(\widehat J))=J$.

\begin{lemma}\label{lem:fixed-endpoint-words}
Let $C$ be a strongly connected component of $G$ and let $1\le i\le s$. For every $\varepsilon>0$ and
all sufficiently large $N$, there are vertices $u_{i,N},v_{i,N}\in C$ and a set
$\mathcal V_{i,N}\subset\cL_N(\tau_i(X_C))$ such that
\[
  \#\mathcal V_{i,N}\ge
  \exp\bigl(N(h_i(C)-\varepsilon)\bigr),
\]
and every word in $\mathcal V_{i,N}$ has a lift that begins at $u_{i,N}$ and ends at $v_{i,N}$.
\end{lemma}

\begin{proof}

Choose one lift for each $J\in\cL_N(\tau_i(X_C))$ and partition $\cL_N(\tau_i(X_C))$ according to
the initial and terminal vertices of the chosen lifts. There are at most $(\#C)^2$ classes,
so one class $\mathcal V_{i,N}$, with common endpoint pair $(u_{i,N},v_{i,N})\in C\times C$, satisfies

\[
  \#\mathcal V_{i,N}
  \ge \frac{\#\cL_N(\tau_i(X_C))}{(\#C)^2}.
\]
By the definition of $h_i(C)$,
\[
  \lim_{N\to\infty}\frac1N\log\#\cL_N(\tau_i(X_C))=h_i(C).
\]
Therefore, for all sufficiently large $N$, we have
\[
\begin{aligned}
  \log\#\mathcal V_{i,N}
  &\ge \log\#\cL_N(\tau_i(X_C))-2\log\#C\\
  &\ge N(h_i(C)-\varepsilon/2)-N\varepsilon/2\\
  &=N(h_i(C)-\varepsilon).
\end{aligned}
\]
This completes the proof.
\end{proof}

\begin{proposition}\label{prop:reducible-box}

Let $X\subset\cA^{\NN}$ be a sofic subshift with the finite vertex-labelled graph
presentation fixed above, let $C_1,\ldots,C_M$ be all its strongly connected components, and
set $K=R(X)$. Then the box dimension of $K$ exists and
\begin{equation}\label{eq:reducible-box-formula}
  \dimB K
  =\frac1{\log n_s}
   \max_{(D_1,\ldots,D_s)}
   \sum_{i=1}^s a_i h_i(D_i),
\end{equation}
where the maximum is over component chains $(D_1,\ldots,D_s)$ with
$D_i\in\{C_1,\ldots,C_M\}$, and $h_i$ is defined in \eqref{eq:component-entropy}.
\end{proposition}

\begin{proof}
For $k\ge1$ and $0\le i\le s$, write
$k_i=\lfloor\theta_i k\rfloor$, and write $m_i=k_i-k_{i-1}$ for $1\le i\le s$. Then
$\sum_{i=1}^s m_i=k$ and $m_i/k\to a_i=\theta_i-\theta_{i-1}$.

Each level-$k$ approximate cube has diameter comparable to $n_s^{-k}$, while every set of
diameter at most $n_s^{-k}$ meets at most a bounded number of level-$k$ approximate cubes. Thus
\eqref{eq:cube-size},
\eqref{eq:encoding-interface}, and \eqref{eq:number-approximate-cubes} show that the
logarithms of the covering number of $K$ at scale $n_s^{-k}$ and of $\#\calD_k(X)$ differ
by $O(1)$. Consequently, the lower and upper box dimensions are obtained from the
corresponding lower and upper limits of $k^{-1}\log\#\calD_k(X)$ by division by
$\log n_s$. It is therefore enough to estimate $\log\#\calD_k(X)$.

\smallskip

Fix $\varepsilon>0$. We first prove the upper bound. Since there are finitely many pairs $(i,C)$, we have,
uniformly for all $N\ge1$, $i$, and $C$,
\begin{equation}\label{eq:component-language-upper}
  \#\cL_N(\tau_i(X_C))
  \lesssim_\varepsilon\exp\bigl(N(h_i(C)+\varepsilon)\bigr).
\end{equation}
For every $\mathbf J=(J_1,\ldots,J_s)\in\calD_k(X)$, choose a word
$I(\mathbf J)\in\cL_k(X)$ such that
\[
  \tau_i\bigl(I(\mathbf J)|_{k_{i-1}}^{k_i}\bigr)=J_i
  \qquad 1\le i\le s.
\]
Since $I(\mathbf J)$ determines $\mathbf J$, the words $I(\mathbf J)$ chosen for distinct
$\mathbf J$ are distinct. For each $I(\mathbf J)$, choose a
vertex path $\widehat I(\mathbf J)\in\cL_k(\Sigma_G)$ with
$\ell(\widehat I(\mathbf J))=I(\mathbf J)$. Decompose this path as
\[
  \widehat I(\mathbf J)=U_0P^{(1)}U_1\cdots U_{q-1}P^{(q)}U_q,
\]
where each $P^{(j)}$ is a maximal vertex subpath contained in a component
$C^{(j)}$, and every vertex in the subpath $U_j$ lies outside all components. A vertex outside the
components cannot occur twice in a directed path, since the segment between two occurrences
would be a directed cycle. A path also cannot leave and later return to the same component,
since the intervening vertices would then be strongly connected to that component. Hence
$C^{(1)},\ldots,C^{(q)}$ are distinct and form a component chain, $q\le M$, and
$
  \sum_{j=0}^q|U_j|\le\#V.
$
For $1\le i\le s$ and $1\le j\le q$, let $n_{i,j}$ be the number of positions in
$(k_{i-1},k_i]\cap\NN$ occupied by $P^{(j)}$. For each $\mathbf J$, form the tuple
\[
  T(\widehat I(\mathbf J))=
  \left(q,(C^{(j)})_{j=1}^q,(U_j)_{j=0}^q,
  (n_{i,j})_{\substack{1\le i\le s\\1\le j\le q}}\right).
\]
Let $\mathfrak T_k=\{T(\widehat I(\mathbf J)):\mathbf J\in\calD_k(X)\}$. The number of possibilities
for $q$, $(C^{(j)})_{j=1}^q$, and $(U_j)_{j=0}^q$ is independent of $k$. Since
$0\le n_{i,j}\le k$,
\begin{equation*}
  \#\mathfrak T_k
  \lesssim\sum_{q=0}^M(k+1)^{sq}
  =\frac{(k+1)^{s(M+1)}-1}{(k+1)^s-1}
  \lesssim(k+1)^{sM}.
\end{equation*}

Fix $T\in\mathfrak T_k$ and define
\[
  \mathcal J_i(T)=\{1\le j\le q:n_{i,j}>0\}.
\]
For all sufficiently large $k$, $m_i>\#V$ for every $i$, so $\mathcal J_i(T)$ is nonempty.
Choose $j_i\in\mathcal J_i(T)$ and set $D_i=C^{(j_i)}$ so that
\[
  h_i(D_i)=\max_{j\in\mathcal J_i(T)}h_i(C^{(j)}).
\]
The component subpaths $P^{(1)},\ldots,P^{(q)}$ occur in this order along $\widehat I(\mathbf J)$, and
the index intervals $(k_{i-1},k_i]\cap\NN$ are ordered from left to right. Hence $j_i\le j_{i+1}$ for $1\le i<s$, and therefore
$(D_1,\ldots,D_s)$ is a component chain. Moreover,
\begin{equation*}
  \sum_{j=1}^q n_{i,j}\le m_i,
  \qquad
  \sum_{j=1}^q n_{i,j}h_i(C^{(j)})
  \le h_i(D_i)\sum_{j=1}^q n_{i,j}
  \le m_i h_i(D_i).
\end{equation*}

For fixed $T$, the vertex words $U_j$ and hence their labels are fixed. The part of $J_i$
corresponding to $P^{(j)}$ is a word in
$\cL_{n_{i,j}}(\tau_i(X_{C^{(j)}}))$. Thus \eqref{eq:component-language-upper} gives
\begin{align*}
 \#\{\mathbf J\in\calD_k(X):T(\widehat I(\mathbf J))=T\}
 &\le
  \prod_{i=1}^s\prod_{\substack{1\le j\le q\\n_{i,j}>0}}
  \#\cL_{n_{i,j}}\bigl(\tau_i(X_{C^{(j)}})\bigr)\\
 &\lesssim_\varepsilon
  \exp\left\{
    \sum_{i=1}^s\sum_{j=1}^q
    n_{i,j}\bigl(h_i(C^{(j)})+\varepsilon\bigr)
  \right\}\\
 &\lesssim_\varepsilon
  \exp\left\{\sum_{i=1}^s m_i h_i(D_i)+\varepsilon k\right\}.
\end{align*}
Summing over $T\in\mathfrak T_k$ gives, for a constant $C_\varepsilon>0$ independent of $k$,
\begin{align*}
  \#\calD_k(X)
  \le C_\varepsilon (k+1)^{sM}
  \times
  \exp\left\{
    \max_{\substack{(D_1,\ldots,D_s)\text{ is}\\
                     \text{a component chain}}}
    \sum_{i=1}^s m_i h_i(D_i)+\varepsilon k
  \right\}.
\end{align*}
Taking logarithms gives
\begin{equation}\label{eq:reducible-box-upper-count}
  \log\#\calD_k(X)
  \le
  \max_{\substack{(D_1,\ldots,D_s)\text{ is}\\
                   \text{a component chain}}}
  \sum_{i=1}^s m_i h_i(D_i)+\varepsilon k+sM\log(k+1)+\log C_\varepsilon.
\end{equation}

For the lower bound, fix a component chain $(D_1,\ldots,D_s)$ and set $B=\#V$.
Since $(D_i,D_{i+1})$ is a component chain, any vertex in $D_i$ can be joined to
any vertex in $D_{i+1}$ by a positive path of length at most $B$.

For all sufficiently large $k$, Lemma~\ref{lem:fixed-endpoint-words}, applied for each
$i$ with $N=m_i-B$, gives vertices $u_i:=u_{i,m_i-B}$ and $v_i:=v_{i,m_i-B}$, and a set
$\mathcal V_i:=\mathcal V_{i,m_i-B}\subset\cL_{m_i-B}(\tau_i(X_{D_i}))$ such that
\begin{equation*}
  \log\#\mathcal V_i\ge(m_i-B)(h_i(D_i)-\varepsilon),
\end{equation*}
and, for every $J_i\in\mathcal V_i$, a chosen vertex-path lift
$\widehat J_i\in\cL_{m_i-B}(\Sigma_{D_i})$ begins at $u_i$ and ends at $v_i$.

For $1\le i<s$, choose a vertex path $v_i\widehat W_i u_{i+1}$ of length at most $B$, and write
$q_i=|\widehat W_i|\le B$. This is independent of $J_i,J_{i+1}$.
Hence every tuple $(J_1,\ldots,J_s)\in\prod_{i=1}^s\mathcal V_i$ gives a vertex path
\[
  \widehat J_1\widehat W_1\widehat J_2\widehat W_2\cdots
  \widehat W_{s-1}\widehat J_s\in\cL(\Sigma_G)
\]
of length
\[
  \sum_{i=1}^s(m_i-B)+\sum_{i=1}^{s-1}q_i
  =k-sB+\sum_{i=1}^{s-1}q_i<k.
\]
The terminal vertex of $\widehat J_s$ is fixed and lies in $D_s$. Starting from this
vertex, extend every such path by the same vertex path until the resulting path has length $k$, and let
$I(J_1,\ldots,J_s)\in\cL_k(X)$ be the word obtained by applying $\ell$ to the resulting path.

For $1\le i\le s$, let
$
  d_i=(i-1)B-\sum_{r=1}^{i-1}q_r.
$
The number of labels preceding $\ell(\widehat J_i)$ is
\[
  \sum_{r=1}^{i-1}(m_r-B)+\sum_{r=1}^{i-1}q_r
  =k_{i-1}-d_i.
\]
Therefore $\ell(\widehat J_i)$ occupies the positions
\[
  (k_{i-1}-d_i,k_i-B-d_i]\cap\NN,
  \qquad 0\le d_i\le(i-1)B.
\]
For sufficiently large $k$, we have $d_i<m_i-B$. The digits of $X_{D_i}$ belong to
$\ell(D_i)$, and the map $\tau_i$ cannot increase their cardinality. Hence
$\tau_i(X_{D_i})$ has at most $\#D_i\le B$ digits, so there are at most $B^{d_i}$ possible
length-$d_i$ prefixes among the words in $\mathcal V_i$. We may therefore choose
$\mathcal V_i'\subset\mathcal V_i$ whose words have the same first $d_i$ digits and such
that
\[
  \#\mathcal V_i'
  \ge B^{-d_i}\#\mathcal V_i
  \ge B^{-(i-1)B}\#\mathcal V_i.
\]

Map each tuple $(J_1',\ldots,J_s')\in\prod_i\mathcal V_i'$ to the vector in
$\calD_k(X)$ determined by $I(J_1',\ldots,J_s')$. For $1\le i\le s$, the $i$-th coordinate
$J_i'$ gives
\[
  \tau_i\bigl(I(J_1',\ldots,J_s')|_{k_{i-1}}^{k_i}\bigr)
  |_{0}^{m_i-B-d_i}=J_i'|_{d_i}^{m_i-B}.
\]
The deleted prefix is fixed on $\mathcal V_i'$, so different choices of $J_i'$ give different
$i$-th coordinates. The map is therefore injective.
It follows that
\begin{align*}
  \log\#\calD_k(X)
  &\ge\sum_{i=1}^s\log\#\mathcal V_i'\\
  &\ge\sum_{i=1}^s(m_i-B)(h_i(D_i)-\varepsilon)
       -B\log B\sum_{i=1}^s(i-1)\\
  &\ge\sum_{i=1}^s m_i h_i(D_i)-\varepsilon k-O(1).
\end{align*}
Taking the maximum over all component chains gives
\begin{equation}\label{eq:reducible-box-lower-count}
  \log\#\calD_k(X)
  \ge
  \max_{\substack{(D_1,\ldots,D_s)\text{ is}\\
                   \text{a component chain}}}
  \sum_{i=1}^s m_i h_i(D_i)-\varepsilon k-O(1).
\end{equation}

Since the number of component chains is finite and $m_i/k\to a_i$,
\eqref{eq:reducible-box-upper-count} and \eqref{eq:reducible-box-lower-count} give, after
$\varepsilon\downarrow0$,
\begin{equation}\nonumber
\begin{aligned}
  \lim_{k\to\infty}\frac1k\log\#\calD_k(X)
  &=\lim_{k\to\infty}
  \frac1k
  \max_{\substack{(D_1,\ldots,D_s)\text{ is}\\
                   \text{a component chain}}}
  \sum_{i=1}^s m_i h_i(D_i)\\
  &=\max_{\substack{(D_1,\ldots,D_s)\text{ is}\\
                   \text{a component chain}}}
  \sum_{i=1}^s a_i h_i(D_i).
\end{aligned}
\end{equation}
Therefore the box dimension of
$K$ exists and is given by \eqref{eq:reducible-box-formula}.
\end{proof}

When $s=2$ and $\Lambda=\diag(n,m)$ with $n>m$, Proposition~\ref{prop:reducible-box}
becomes
\begin{equation}\label{eq:fraser-jurga-sofic}
  \dimB R(X)
  =\max_C\left\{
    \frac{h_{\mathrm{top}}(X_C)}{\log n}
    +\max_{\substack{D:\ (C,D)\text{ is}\\
                     \text{a component chain}}}
      h_{\mathrm{top}}\bigl(\tau_2(X_D)\bigr)
      \left(\frac1{\log m}-\frac1{\log n}\right)
  \right\}.
\end{equation}
This agrees with the planar formula of Fraser and Jurga
\cite[Theorem~1.2]{FraserJurga2023}.

\subsection{The case \texorpdfstring{$\dimH K=\dimB K$}{dimH K = dimB K}}

Proposition~\ref{prop:reducible-box} shows that the box dimension exists. We now prove that
dimension equality implies positivity and $\sigma$-finiteness. The critical Hausdorff
measure, however, need not be finite.

\begin{lemma}\label{lem:reducible-equality}
Let $X$ be a sofic subshift, set $K=R(X)$, and suppose that
$\dimH K=\dimB K=\gamma$. Then $\HH^\gamma(K)>0$ and $\HH^\gamma$ is $\sigma$-finite on
$K$.
\end{lemma}

\begin{proof}
Fix a finite vertex-labelled graph presentation of $X$, and let $C_1,\ldots,C_M$ be all its
strongly connected components. For every component $C_j$, write
$\delta_j=\dimH K_{C_j}$. By \eqref{eq:reducible-Hdim},
\[
  \gamma=\max_{1\le j\le M}\delta_j.
\]
If $\delta_j=\gamma$, then $K_{C_j}\subset K$ and monotonicity of box dimensions give
\[
  \gamma=\dimH K_{C_j}
  \le \underline{\dim}_{\mathrm B}K_{C_j}
  \le \overline{\dim}_{\mathrm B}K_{C_j}
  \le \dimB K=\gamma.
\]

Hence the Hausdorff and box dimensions of $K_{C_j}$ are both equal to $\gamma$.
Theorem~\ref{thm:main} therefore gives
\begin{equation}\label{eq:maximal-component-finite-measure}
  0<\HH^\gamma(K_{C_j})<\infty
  \qquad\text{whenever }\delta_j=\gamma.
\end{equation}
If $\delta_j<\gamma$, the definition of Hausdorff dimension gives
\begin{equation}\label{eq:smaller-component-zero-measure}
  \HH^\gamma(K_{C_j})=0.
\end{equation}
It follows from \eqref{eq:maximal-component-finite-measure} and
\eqref{eq:smaller-component-zero-measure} that every component set has finite
$\gamma$-dimensional Hausdorff measure. At least one component has
$\delta_j=\gamma$, so \eqref{eq:maximal-component-finite-measure} and
$K_{C_j}\subset K$ also give $\HH^\gamma(K)>0$.

Each map $f_u$ in \eqref{eq:scc-decomposition} is a Lipschitz affine map. Hence
\eqref{eq:scc-decomposition} and the finiteness of the component measures proved above show
that $K$ is a countable union of Borel sets of finite $\HH^\gamma$ measure. Thus
$\HH^\gamma$ is $\sigma$-finite on $K$.
\end{proof}

We now show that the conclusion of Lemma~\ref{lem:reducible-equality} cannot in
general be strengthened to finiteness. For an integer $m\ge2$ and a nonempty digit set $D\subseteq\{0,\ldots,m-1\}$, write
\[
E_m(D)=
\left\{\sum_{j=1}^{\infty}m^{-j}d_j:d_j\in D\right\}.
\]

\begin{example}\label{ex:reducible-equality-infinite}
Consider
\[
  \Lambda=\diag(6,3),
  \qquad
  \cS_A=\{(3,2),(5,2)\},
  \qquad
  \cS_B=\{(0,0),(1,0)\}.
\]
Let $\cS=\cS_A\cup\cS_B$. Let $X\subset\cS^{\NN}$ be the subshift of finite type in which every
transition inside $\cS_A$, every transition inside $\cS_B$, and every transition from
$\cS_A$ to $\cS_B$ is allowed, while no transition from $\cS_B$ to $\cS_A$ is allowed. Set
\[
  X_A=\cS_A^{\NN},\qquad X_B=\cS_B^{\NN},\qquad
  K=R(X),\quad K_A=R(X_A),\quad K_B=R(X_B).
\]
The coding map gives
\[
  K_A=E_6(\{3,5\})\times\{1\},
  \qquad
  K_B=E_6(\{0,1\})\times\{0\}.
\]
Consequently, for
$
  \gamma=\frac{\log2}{\log6},
$
we have
\[
  \dimH K_A=\dimB K_A=\dimH K_B=\dimB K_B=\gamma,
  \qquad
  0<\HH^\gamma(K_A),\HH^\gamma(K_B)<\infty.
\]
Combining this with equation~\eqref{eq:reducible-Hdim} gives $\dimH K=\gamma$.
The component entropies are
$h_{\mathrm{top}}(X_A)=h_{\mathrm{top}}(X_B)=\log2$, while both projected component
entropies are zero. Formula \eqref{eq:fraser-jurga-sofic} therefore gives
\[
  \dimB K=\frac{\log2}{\log6}=\gamma=\dimH K.
\]

It remains to compute the measure. For $r\ge1$, let
\[
  Y_r=\{I\omega:I\in\cS_A^r,\ \omega\in X_B\}.
\]
Set $c_B=\HH^\gamma(K_B)$. For each
$I\in\cS_A^r$, the set $R_r(I)+\Lambda^{-r}K_B$ is a copy of $K_B$ on which distances are
multiplied by $6^{-r}$. Since $6^{-r\gamma}=2^{-r}$,
\begin{equation}\label{eq:equality-example-copy-measure}
  \HH^\gamma\bigl(R_r(I)+\Lambda^{-r}K_B\bigr)=2^{-r}c_B.
\end{equation}
Their first-coordinate cylinder sets corresponding to $I\in\cS_A^r$ are separated, so for fixed $r$ these $2^r$ copies are
pairwise disjoint.
It follows from \eqref{eq:equality-example-copy-measure} that
\begin{equation}\label{eq:equality-example-layer-measure}
  \HH^\gamma(R(Y_r))=2^r2^{-r}c_B=c_B.
\end{equation}

The set $R(Y_r)$ lies on the horizontal line with second coordinate $1-3^{-r}$, so the Borel
sets $R(Y_r)$ are pairwise disjoint as $r$ varies. Therefore countable additivity and
\eqref{eq:equality-example-layer-measure} give
\[
  \HH^\gamma(K)
  \ge \sum_{r=1}^{\infty}\HH^\gamma(R(Y_r))
  =\sum_{r=1}^{\infty}c_B
  =\infty.
\]

\end{example}

\subsection{The case \texorpdfstring{$\dimH K<\dimB K$}{dimH K < dimB K}}

We give two examples with $\dimH K<\dimB K$. In the first, the critical Hausdorff
measure is positive and finite; in the second, it is infinite but $\sigma$-finite.

\begin{example}\label{ex:finite-gap-critical-measure}

Consider
\[
  \Lambda=\diag(9,3),
  \qquad
  \theta_1=\frac{\log3}{\log9}=\frac12.
\]
Take two disjoint full-shift digit sets
\[
  \cS_A=\{0,1,\ldots,6\}\times\{1,2\},\qquad
  \cS_B=(\{0,1,\ldots,8\}\times\{0\})
  \cup\{(7,1),(7,2)\}.
\]
Let $\cS=\cS_A\cup\cS_B$. Let $X\subset\cS^{\NN}$ be the subshift of finite type in which every
transition inside $\cS_A$, every transition inside $\cS_B$, and every transition from
$\cS_A$ to $\cS_B$ is allowed, while no transition from $\cS_B$ to $\cS_A$ is allowed. Set
\[
  X_A=\cS_A^{\NN},\qquad X_B=\cS_B^{\NN},\qquad
  K=R(X),\quad K_A=R(X_A),\quad K_B=R(X_B).
\]
The Bedford-McMullen dimension formulas \cite{Bedford1984,McMullen1984} give
\begin{align*}
  \dimH K_A=\dimB K_A
  &=\frac{\log(2\sqrt7)}{\log3}\approx1.51655,\\
  \dimH K_B&=\frac{\log5}{\log3}\approx1.46497,\\
  \dimB K_B&=\frac{\log\sqrt{33}}{\log3}\approx1.59133.
\end{align*}
Set $\gamma=\log(2\sqrt7)/\log3$. The identity
$K_A=E_9(\{0,\ldots,6\})\times E_3(\{1,2\})$ shows that $K_A$ is
$\gamma$-Ahlfors regular, and hence $0<\HH^\gamma(K_A)<\infty$. Since
$\dimH K_B<\gamma$, every affine image of $K_B$ on the right-hand side of
\eqref{eq:scc-decomposition} has zero $\HH^\gamma$ measure, while the corresponding images
of $K_A$ are contained in $K_A$. Therefore \eqref{eq:scc-decomposition} gives
\[
  \HH^\gamma(K)=\HH^\gamma(K_A)\in(0,\infty).
\]

The box dimension follows directly from \eqref{eq:fraser-jurga-sofic}. The component and
projected entropies are
\[
\begin{alignedat}{2}
  h_{\mathrm{top}}(X_A)&=\log14,\qquad
  &h_{\mathrm{top}}(\tau_2(X_A))&=\log2,\\
  h_{\mathrm{top}}(X_B)&=\log11,
  &h_{\mathrm{top}}(\tau_2(X_B))&=\log3.
\end{alignedat}
\]
Hence \eqref{eq:fraser-jurga-sofic} gives
\begin{equation*}
\begin{aligned}
  \dimB K
  &=\max\left\{
    \frac{\log14}{\log9}
      +\log3\left(\frac1{\log3}-\frac1{\log9}\right),
    \frac{\log11}{\log9}
      +\log3\left(\frac1{\log3}-\frac1{\log9}\right)
  \right\}\\
  &=\frac{\log42}{\log9}.
\end{aligned}
\end{equation*}
Since $\gamma=\log28/\log9$, we obtain
\[
  \dimH K=\frac{\log28}{\log9}
  <\frac{\log42}{\log9}=\dimB K,
  \qquad
  0<\HH^\gamma(K)<\infty.
\]
\end{example}

\begin{example}\label{ex:infinite-gap-critical-measure}
Consider
\[
  \Lambda=\diag(16,4),
  \qquad
  \cS_A=\{(8,2),(10,2),(12,2),(14,2)\},
  \qquad
  \cS_B=\{(1,0),(1,1)\}.
\]
Let $\cS=\cS_A\cup\cS_B$. Let $X\subset\cS^{\NN}$ be the subshift of finite type in which every
transition inside $\cS_A$, every transition inside $\cS_B$, and every transition from
$\cS_A$ to $\cS_B$ is allowed, while no transition from $\cS_B$ to $\cS_A$ is allowed.
Set
\[
  X_A=\cS_A^{\NN},
  \qquad
  X_B=\cS_B^{\NN},
  \qquad
  K=R(X),
  \qquad
  K_A=R(X_A),\quad K_B=R(X_B).
\]
The definition of $R$ gives
\[
  K_A=E_{16}(\{8,10,12,14\})\times\left\{\frac23\right\},
  \qquad
  K_B=\left\{\frac1{15}\right\}\times E_4(\{0,1\}).
\]
Both one-dimensional systems satisfy the strong separation condition. Consequently,
\[
  \dimH K_A=\dimB K_A=\frac{\log4}{\log16}=\frac12,
  \qquad
  \dimH K_B=\dimB K_B=\frac{\log2}{\log4}=\frac12,
\]
and
\begin{equation}\label{eq:transition-component-measures}
  0<\HH^{1/2}(K_A)<\infty,
  \qquad
  0<\HH^{1/2}(K_B)<\infty.
\end{equation}
Equation \eqref{eq:reducible-Hdim} gives $\dimH K=1/2$.

For $r\ge1$, let
\[
  Y_r=\{I\omega:I\in\cL_r(X_A),\ \omega\in X_B\}.
\]
Set $c_B=\HH^{1/2}(K_B)$. Since $X_A$ is the full shift on four digits,
$\#\cL_r(X_A)=4^r$.
For $I\in\cL_r(X_A)$, the set $R_r(I)+\Lambda^{-r}K_B$ varies only in the second
coordinate, where its similarity ratio is $4^{-r}$. Thus
\[
  \HH^{1/2}\bigl(R_r(I)+\Lambda^{-r}K_B\bigr)
  =2^{-r}c_B.
\]
Their first-coordinate cylinder sets are separated, so for fixed $r$ these $4^r$ copies are
pairwise disjoint, and hence
$\HH^{1/2}(R(Y_r))=4^r2^{-r}c_B=2^rc_B$. Since $R(Y_r)\subset K$,
\[
  \HH^{1/2}(K)\ge2^r c_B\qquad r\ge1.
\]
Letting $r\to\infty$ gives $\HH^{1/2}(K)=\infty$. On the other hand,
\eqref{eq:scc-decomposition} and \eqref{eq:transition-component-measures} show that $K$ is
covered by countably many Borel sets of finite $\HH^{1/2}$ measure. Thus
$\HH^{1/2}$ is $\sigma$-finite on $K$.

The box dimension follows from \eqref{eq:fraser-jurga-sofic}. We have
\[
\begin{alignedat}{2}
  h_{\mathrm{top}}(X_A)&=\log4,\qquad
  &h_{\mathrm{top}}(\tau_2(X_A))&=0,\\
  h_{\mathrm{top}}(X_B)&=\log2,
  &h_{\mathrm{top}}(\tau_2(X_B))&=\log2.
\end{alignedat}
\]
Hence \eqref{eq:fraser-jurga-sofic} gives
\[
\begin{aligned}
  \dimB K
  &=\max\left\{
    \frac{\log4}{\log16}
      +\log2\left(\frac1{\log4}-\frac1{\log16}\right),
    \frac{\log2}{\log16}
      +\log2\left(\frac1{\log4}-\frac1{\log16}\right)
  \right\}\\
  &=\frac34.
\end{aligned}
\]
Therefore
\[
  \dimH K=\frac12<\frac34=\dimB K,
  \qquad
  \HH^{1/2}(K)=\infty,
\]
and this measure is $\sigma$-finite on $K$.
\end{example}

\begin{proof}[Proof of Theorem~\ref{thm:reducible}]
Proposition~\ref{prop:reducible-box} proves that the box dimension exists and gives its
formula. Suppose first that $\dimH K=\dimB K=\gamma$.
Lemma~\ref{lem:reducible-equality} gives $\HH^\gamma(K)>0$ and shows that
$\HH^\gamma$ is $\sigma$-finite on $K$. Example~\ref{ex:reducible-equality-infinite} shows
that $\HH^\gamma(K)$ may nevertheless be infinite. This proves \textup{(i)}.

For \textup{(ii)}, Example~\ref{ex:finite-gap-critical-measure} gives a dimension gap
with positive and finite critical Hausdorff measure, whereas
Example~\ref{ex:infinite-gap-critical-measure} gives a dimension gap with infinite but
$\sigma$-finite critical Hausdorff measure.
\end{proof}

\end{document}